\documentclass[leqno,11pt,oneside]{amsart}
\usepackage{amsmath,amssymb,amsthm,mathrsfs}
\usepackage{epsfig,color}
\usepackage[latin1,utf8]{inputenc}
\usepackage[english]{babel}
\usepackage{mathtools}
\usepackage{braket}
\usepackage[toc]{appendix}
\usepackage{enumitem}
\usepackage{ulem}
\usepackage{graphicx}
\usepackage{xfrac, nicefrac}
\usepackage{csquotes}
\usepackage{todonotes}
\allowdisplaybreaks

\numberwithin{equation}{section}

\usepackage{esint}

\usepackage{hyperref}

\newcommand{\rn}{\mathbb{R}^n}

\def\A{\mathcal A}

\def\H{\mathcal H}

\def\R{\mathbb R}
\def\N{\mathbb N}

\def\e{\varepsilon}

\def\vphi{\varphi}

\def\l{\lambda}

\def\Om{\Omega}

\renewcommand{\d}{\mathrm{d}}

\renewcommand{\l}{\lambda}
\renewcommand{\L}{\Lambda}

\newcommand{\ov}{\overline}

\def\C{\mathbf{C}}

\def\B{\mathcal{B}}

\def\A{\mathcal{A}}
\def\d{\delta}

\newtheorem*{theorem*}{Theorem}
\newtheorem{theorem}{Theorem}[section]
\newtheorem{lemma}[theorem]{Lemma}

\newtheorem*{proposition*}{Proposition}

\newtheorem{remark}[theorem]{Remark}
\newtheorem*{remark*}{Remark}

\title[Global Lipschitz regularity for anisotropic elliptic problems]{Global Lipschitz regularity for anisotropic elliptic problems with natural gradient growth 
}

\begin{document}

\begin{abstract}
Homogeneous Dirichlet and Neumann boundary-value problems for anisotropic elliptic equations of $p$-Laplace type are considered. They emerge as Euler-Lagrange equations of integral functionals of the Calculus of Variations built upon possibly anisotropic norms of the gradient of trial functions. Global Lipschitz regularity of solutions  is established under weakest possible integrability assumptions on  the right-hand side, which can also include lower order terms with natural growth in the gradient.
 The results hold in either convex domains, or domains enjoying minimal integrability assumptions on the curvature of their boundary.
\end{abstract}

\author{Carlo Alberto Antonini \textsuperscript{1}}
\address{\textsuperscript{1}  Istituto Nazionale di Alta Matematica ``Francesco Severi'' (INdAM) and
Dipartimento di Matematica e Informatica ``Ulisse Dini'',
Universit\`a di Firenze,
Viale Morgagni 67/A, 50134
Firenze,
Italy}
\email{antonini@altamatematica.it}
\urladdr{0000-0002-7663-1090}

\author{Andrea Cianchi\textsuperscript{2}}
\address{\textsuperscript{2}Dipartimento di Matematica e Informatica ``Ulisse Dini'',
Universit\`a di Firenze,
Viale Morgagni 67/A, 50134
Firenze,
Italy}
\email{andrea.cianchi@unifi.it}
\urladdr{0000-0002-1198-8718}


\subjclass[2020]{35J25, 35J60}
\keywords{Anisotropic elliptic equations, $p$-Laplacian, Orlicz-Laplacian,
 natural gradient growth, Dirichlet problems, Neumann problems, convex domains}

\maketitle

\section{Introduction}
This work deals with boundary value problems for elliptic equations of the form
\begin{equation}\label{equation}
    - \Delta_p^H u = g(x,u,\nabla u) \qquad \text{in $\Omega$,}
\end{equation}
where:
\\
(i) $\Omega$ is a bounded open set in $\rn$, with $n\geq 2$;
\\
(ii) $\Delta_p^H$ is the anisotropic $p$-Laplace operator, associated with an exponent $p>1$ and a norm $$H : \rn \to [0,\infty),$$  defined as 
\begin{equation}\label{def:anpl}
    \Delta_p^H u= \tfrac{1}{p}\mathrm{div}\big( \nabla_\xi H^p(\nabla u)\big)
\end{equation}
for a function $u:\Omega\to \R$;
\\ (iii) $g: \Omega \times \R \times \rn \to \R$ is a Carath\'eodory function such that, for a.e. $x\in \Omega$,
\begin{align}
    \label{g}
    \kappa H(\xi)^p - f(x) \leq g(x,t,\xi) \leq \kappa H(\xi)^p + f(x) \quad \text{for $t\in \R$ and $\xi \in \rn$,}
\end{align}
for some constant  $\kappa\in \R$ and some function $f: \Omega \to [0, \infty)$.

Minimal assumptions on the domain $\Omega$, on the norm $H$, and on the function $f$ are exhibited for any bounded weak solution $u$ to the equation \eqref{equation}, subject to  homogeneous Dirichlet boundary conditions  
or  co-normal Neumann conditions,
to be globally Lipschitz continuous on $\Omega$. Thanks to the regularity to be imposed on $\Omega$, the latter property of $u$ is equivalent to the global boundedness of $|\nabla u|$.

The a priori boundedness assumption on $u$ is indispensable,
due to the critical   --  also called \lq\lq natural" according to a standard terminology -- $p$-growth in the gradient of the right-hand side of the equation \eqref{equation}. It is indeed well-known that the gradient of an unbounded weak solution to the Dirichlet problem can be unbounded, even in the special case when $\Delta_p^H=\Delta$, the plain Laplace operator corresponding to the choice $p=2$ and $H(\xi)=|\xi|$  --   see \cite{frehse}. 

The literature on anisotropic elliptic equations is quite rich. Especially, equations involving the anisotropic  $p$-Laplace operator  \eqref{def:anpl} have been analyzed along diverse directions in the last decades. A sample of contributions on this topic   includes   \cite{al:al, acf, accfm, CHN, cisa1,cfr20, dpv, EMSV, FSV, fe:ka,  mos:mar, MST, Sp:V}. 
 
 Existence results for solutions to boundary value problems for elliptic equations with natural gradient growth in the lower order term 
  are the subject of \cite{ADP, femu, figuer, kaz,ser}. A renewed interest in equations of this kind    arose 
in connection with the level set formulation of the inverse mean curvature flow -- see \cite{rmoser}, and   \cite{rivas,fdp, mari} where the anisotropic case is treated.

 $L^\infty$-bounds for the gradient of solutions are a central issue in the regularity theory of elliptic partial differential equations and a fundamental step in the proof of further properties. The local Lipschitz continuity of solutions  to classes of equations   with natural gradient growth, including \eqref{equation}, has been investigated in  nowadays classical papers, such as  
 \cite{diben, lieb91,tol}.
 The global Lipschitz continuity of the solutions to homogeneous Dirichlet and Neumann problems associated with equations of the form \eqref{equation} is established in \cite{lieb88} under standard conditions on $\Omega$ and $f$.
As mentioned above, the punctum of our work is the detection of  weakest possible  assumptions on $\Omega$ and $f$ for this property of solutions to hold.
The present research complements 
earlier results in the same spirit from \cite{cia11} (see also \cite{cia14} dealing with systems) under the following respects:

\medskip
\par\noindent
$\bullet$ The dependence on $\nabla u$  of the elliptic operator is through a general norm $H$;

\medskip
\par\noindent
$\bullet$  Right-hand sides of the equation also depending on $\nabla u$, with the natural growth, are allowed.

\medskip

Our hypotheses on $H$, $f$, and $\Omega$ read as follows.
We assume  
 that $H^2\in C^2(\R^n\setminus\{0\})$ and fulfills the ellipticity condition
 \begin{equation}\label{ell:H}
    \l\,|\eta|^2\leq \tfrac{1}{2} \nabla^2_\xi H^2(\xi)\,\eta \cdot \eta \leq \L\,|\eta|^2\quad\text{for $\xi\neq 0$, and $\eta\in\rn$,}
\end{equation}
for some positive constants $\l$ and $\L$. Notice that the essence of \eqref{ell:H}
is the lower estimate, the upper one being a consequence  of the  regularity of $H^2$.
\\ The lower estimate is equivalent to the condition that
\begin{equation}\label{H:unifconv}
 \{H(\xi)<1\}\text{ is uniformly convex},
\end{equation}
where uniform convexity of a convex set in $\rn$ means that  all the principal curvatures of its boundary are bounded away from zero.
\\
The condition  \eqref{ell:H} is satisfied, for example, by any function $H$ of the form
\begin{equation*}
    H(\xi)=\big( \alpha \,K^p(\xi)+\beta \,|\xi|^p\big)^{1/p},
\end{equation*}
for some norm $K\in C^2(\rn\setminus\{0\})$, and some $p>1$ and $\alpha, \beta >0$ -- see \cite{cfv} and \cite[Example 2.1]{acf}. 
In particular, the norm
\begin{equation*}
    H(\xi)=\bigg( \Big(\sum _{i=1}^n |\xi_i|^q\Big)^{p/q}+\,|\xi|^p\bigg)^{1/p}
\end{equation*}
 is admissible for any $p>1$ and any $q \geq 2$.
\\ Further instances of admissible norms $H$ are  the Minkowski gauges of  bounded,  uniformly convex sets $E$, symmetric about the origin, and such that $\partial E \in C^2$. Namely,
\begin{equation*}
   H(\xi) = \inf \{r>0\,:\, \xi\in r\,E \}.
\end{equation*}
This is a consequence of the fact that $\{H<1\} = E$.

The function $f$ in \eqref{g}, governing the growth in the $x$-variable,  is subject to an appropriate integrability condition. Specifically, we assume that either $f\in  L^{n,1}(\Omega)$ or $f \in (L^2)^{(1,\frac{1}{2})}(\Omega)$, according to whether $n\geq 3$ or $n=2$. Here, $L^{q,r}(\Omega)$ and $L^{(q,r)}(\Omega)$  denote, for $1\le q <\infty$ and $1\leq r\leq \infty$, the Lorentz spaces whose norms are defined via the decreasing rearrangement, and the maximal function  of the decreasing rearrangement, respectively. Recall that these two families of spaces agree, up to equivalent norms, if $1<q<\infty$ and $1\leq r <\infty$. Moreover, $(L^2)^{(1,\frac{1}{2})}(\Omega)$ stands for the   set of those functions $f$ such that $f^2\in L^{(1,\frac{1}{2})}(\Omega)$. For brevity, we define
\begin{equation}\label{X:f}
    X_{n}(\Omega)=
    \begin{cases}
        L^{n,1}(\Omega)\quad & \text{if $n\geq 3$}
        \\

        (L^2)^{(1,\frac{1}{2})}(\Omega)\quad & \text{if $n=2$.}
    \end{cases}
\end{equation}

Finally, Lipschitz domains $\Omega$ whose boundary $\partial \Omega$ is endowed with weak curvatures enjoying a sufficient degree of integrability are allowed. This is prescribed by requiring that the relevant curvatures belong to the Lorentz space $L^{(n-1,1)}(\partial \Omega)$, with respect to the $(n-1)$--dimensional Hausdorff measure $\mathcal H^{n-1}$ on $\partial \Omega$. As mentioned above, $L^{(n-1,1)}(\partial \Omega)=L^{n-1,1}(\partial \Omega)$ if $n \geq 3$. On the other hand, if $n=2$, then $L^{(1,1)}(\partial \Omega)= L\log L(\partial \Omega)$, a Zygmund space, which is in turn a special Orlicz space.
Precisely, the
 notation $\partial \Omega\in W^2L^{(n-1,1)}$ means that the boundary of $\Omega$ can be locally represented as the graph of a Lipschitz function of $(n-1)$ variables endowed with second-order weak derivatives in the space $L^{(n-1,1)}$. This amounts to requiring that the second fundamental form $\B_\Omega$ of $\partial \Omega$ belongs to $L^{(n-1,1)}(\partial \Omega)$.

Our main result about the Dirichlet problem
\begin{equation}\label{Ex:dir}
    \begin{cases}
        -\Delta_p^H u= g(x,u, \nabla u)\quad & \text{in $\Omega$}
        \\
        u=0\quad & \text{on $\partial \Omega$,}
    \end{cases}
\end{equation}
and the Neumann problem
\begin{equation}\label{Ex:neu}
\begin{cases}
     -\Delta_p^H u= g(x,u, \nabla u)\quad & \text{in $\Omega$}
\\
\nabla_\xi H^p(\nabla u)\cdot \nu=0\quad & \text{on $\partial \Omega$,}
\end{cases}
\end{equation}
under the boundary regularity on $\Omega$ described above is stated in Theorem \ref{thm:example1} below. Here, $\nu$ denotes the outer unit normal on $\partial \Omega$.
 Recall that  a weak solution to the problem \eqref{Ex:dir} is a function 
 $u\in W^{1,p}_0(\Omega)$ such that
\begin{equation}\label{eqweak}
    \int_{\Omega}\tfrac{1}{p}\nabla_\xi H^p(\nabla u)\cdot \nabla \varphi\,dx=\int_\Omega g(x,u, \nabla u)\,\varphi\,dx
\end{equation}
for every $\varphi\in W^{1,p}_0(\Omega)\cap L^\infty(\Omega)$. A weak solution    to \eqref{Ex:neu} is a function $u\in W^{1,p}(\Omega)$  such that Equation \eqref{eqweak} holds for every $\varphi\in W^{1,p}(\Omega)\cap L^\infty(\Omega)$.

\begin{theorem}[Global Lipschitz regularity in minimally regular domains]\label{thm:example1}
    Let $\Omega$ be a bounded Lipschitz domain in $\rn$, $n\geq 2$, such that $\partial \Omega \in W^2L^{(n-1,1)}$. Let  $H$   be a norm in $\rn$ such that $H^2\in C^2(\R^n\setminus\{0\})$ and fulfilling \eqref{ell:H}. Assume that
    either
    $n\geq 3$ and $f\in  L^{n,1}(\Omega)$, or $n=2$ and $f \in (L^2)^{(1,\frac{1}{2})}(\Omega)$.  
    Let $u$ be a bounded weak solution to the Dirichlet problem \eqref{Ex:dir} or the Neumann problem \eqref{Ex:neu}.
    Then $u$ is Lipschitz continuous on $\Omega$, and there exists a constant $c=c(n,p,\l,\L, \Omega)$ such that
    \begin{equation}\label{stima:ex1}
        \|\nabla u\|_{L^\infty(\Omega)}\leq c\,e^{{\frac{p|\kappa|}{(p-1)^2}}\|u\|_{L^\infty(\Omega)}}\,\|f\|^{\frac{1}{p-1}}_{X_n(\Omega)}
    \end{equation}
    where $X_n(\Omega)$ is defined as \eqref{X:f}.
  In the case when $\kappa=0$, the assumption that $u$  be bounded can be dropped.
\end{theorem}

\begin{remark}
    \label{opt-om}{\rm
In the light of counterexamples   from  \cite{cia141},  the integrability assumption $\partial \Omega \in W^2L^{(n-1,1)}$  imposed on the curvatures of $\partial \Omega$ in Theorem \ref{thm:example1} is essentially sharp. }
\end{remark}
\begin{remark}{\rm
    \label{opt-f} The optimality of the  space $L^{n,1}(\Omega)$ describing the integrability of the function $f$, for $n\geq 3$,
 can be demonstrated, for $n\geq 3$, by the model case of the Poisson equation  \cite{cia92}. We do not know whether the space $(L^2)^{(1,\frac{1}{2})}(\Omega)
$ for $f$ is optimal, for nonlinear operators, when $n=2$. One has that 
\begin{equation*}
    L^{2+\varepsilon}(\Omega)\subsetneq L^2\log^{\gamma}(\Omega)\subsetneq (L^2)^{(1,\frac{1}{2})}(\Omega)\subsetneq L^{2,1}(\Omega)
\end{equation*}
for all $\varepsilon>0$, and $\gamma>2$ (see e.g. \cite[Theorem 10.3.12, Theorem 10.3.20]{pick}).
Hence, our assumption that $f\in (L^2)^{(1,\frac{1}{2})}(\Omega)$ is slightly stronger
than requiring $f\in L^{2,1}(\Omega)$. Let us just mention that in \cite{kus:min1} the latter assumption (in local form) was shown to be sufficient for the local boundedness of the gradient of solutions to isotropic $p$-Laplacian type equations when $g(x,u, \nabla u)=f(x)$.
\\
Several more recent works have addressed the Lipschitz regularity of solutions under the assumption $f \in L^{n,1}$. They include, among others,  \cite{baroni, beck:min, defil:min,defil:pic,dzhu,duz:min,kus:min,ok}.
}
\end{remark}

In the special case when $\Omega$ is convex, the same conclusions as in Theorem \eqref{thm:example1} hold without any additional assumption on $\partial \Omega$. Loosely speaking, this is possible owing to the fact that the second fundamental form of $\partial \Omega$ is globally nonnegative.
In this case, we can also be more specific about the dependence of the constant $c$ on $\Omega$ in the inequality \eqref{stima:ex1}. This dependence is just through its Lebesgue measure $|\Omega|$, its diameter $d_\Omega$, and its Lipschitz characteristics $\mathcal{L}_\Omega=(L_\Omega,R_\Omega)$. The latter  depends on the Lipschitz constant $L_\Omega$ of the functions locally describing $\partial \Omega$ and on the radius $R_\Omega$ of their ball domains, and is defined in   Section \ref{sec:prelim0}.

\begin{theorem}[Global Lipschitz regularity in convex domains]\label{thm:example}
    Let $\Omega$ be a bounded open convex set in $\rn$, with $n \geq2$. Assume that $H$ and $f$ are as in Theorem \ref{thm:example1}.
    Let $u$ be a bounded weak solution to either problem \eqref{Ex:dir} or \eqref{Ex:neu}.  Then $u$ is Lipschitz continuous on $\Omega$, and the inequality \eqref{stima:ex1}  holds for some constant  $c=c(n,p,\l, \L, |\Omega|, d_\Omega, \mathcal{L}_\Omega)$. If $\kappa =0$, then   the assumption that $u$
 be bounded can be dropped.    
\end{theorem}

Thanks to an ad hoc substitution, the original equation \eqref{equation} can be turned into an equation with a right-hand side independent of the gradient of the new unknown, and, thanks the assumption \eqref{g}, enjoying the same integrability properties as the function $f$. Approaching the new problem requires approximating the differential operator. Like this new problem, the approximating problems 
contain   right-hand sides only depending on $x$, but involve
more general elliptic operators driven by an   anisotropic nonlinearity of  Orlicz type. The Lipschitz continuity of solutions to anisotropic  problems in a general Orlicz setting is a question of independent  interest and is the content of the separate statements of Theorems \ref{thm:dirneu} -- \ref{thm:conv}. 
  These theorems extend earlier results from \cite{cia11} dealing with isotropic equations. The anisotropy of the problems under consideration calls for new ad hoc integral identities and inequalities involving 
solutions. In particular, they involve an anisotropic second fundamental form on the boundary associated with the norm $H$.

\section{Background}\label{sec:prelim0}

\subsection{Young functions}
A Young function $B:[0,\infty)\to [0,\infty]$ is a convex function vanishing at $0$. Hence, 
\begin{equation}\label{B}
B(t) = \int_0^t b(s) \,ds \qquad \text{for $t\geq 0$,}
\end{equation}
for some non-decreasing function $b:  [0, \infty) \to [0, \infty]$.
\\
Real-valued Young functions $B\in C^2(0,\infty)$ with a nonlinear growth have a role in our discussion. Specifically, on setting
\begin{equation}\label{dic102}
    i_b=\inf_{t>0}\frac{t\,b'(t)}{b(t)}\quad\text{and }\quad s_b=\sup_{t>0}\frac{t\,b'(t)}{b(t)},
\end{equation}
the nonlinear growth condition on the function $B$ is imposed by requiring that 
\begin{equation}\label{ib}
i_b>0.
\end{equation}
The property \eqref{ib} is equivalent to the so-called $\nabla_2$-condition in the theory of Young functions. The $\Delta_2$-condition is also required on $B$, and is equivalent to 
\begin{equation}\label{sb}
s_b<\infty.
\end{equation}
 By the assumption \eqref{ib},
\begin{equation}\label{limb}
\lim _{t\to 0^+}b(t)=0.
\end{equation}
The monotonicity of the function $b$ ensures that
\begin{equation} \label{Bb}
\tfrac t2 b(\tfrac t2)\leq B(t)  \leq b(t) t \quad \text{for $t \geq 0$.}
\end{equation}
Also, as a consequence of the assumption \eqref{dic102},   the functions $\frac{b(t)}{t^{i_b}}$ and $\frac{b(t)}{t^{s_b}}$
are non-decreasing and non-increasing, respectively.
Hence,
\begin{equation}\label{usual:b}
    b(1)\,\min\{t^{i_b},t^{s_b}\}\leq b(t)\leq b(1)\,\max\{t^{i_b},t^{s_b}\}\quad \text{for $t\geq 0$,}
\end{equation}
and there exist positive constants $c$ and $C$, depending only on $i_b$ and $s_b$, such that
\begin{equation}\label{usual:b1}
    c\,b(s)\leq b(t)\leq C\,b(s) \quad \text{if $0\leq s\leq t\leq 2s$.}
\end{equation}
Now, define the function $a: (0, \infty) \to [0, \infty)$ as
 $$a(t) = \frac{b(t)}t \qquad \text{for $t>0$.}$$
Thereby, $a\in C^1(0, \infty)$. Furthermore, if $i_a$ and $s_a$ are defined in analogy to \eqref{dic102}, then
%
\begin{equation}\label{dic105}
\text{$i_b= i_a+1$\quad  and \quad $s_b=s_a+1$.}
\end{equation}
Hence, the assumptions \eqref{ib} and \eqref{sb} are equivalent to 
\begin{equation}\label{c:a}
    -1<i_a\leq s_a<\infty\,,
\end{equation}
and a counterpart of \eqref{usual:b1} holds;  namely,
\begin{equation}\label{usual:a1}
    c\,a(s)\leq a(t)\leq C\,a(s)\quad\text{if $0<s\leq t\leq 2s$.}
\end{equation}
Also 
\begin{equation}\label{jan201}
\text{$i_B\geq i_b+1$ \quad and \quad  $s_B\leq s_b +1$.}
\end{equation}
Thus, Equations \eqref{ib}-\eqref{sb} imply that $i_B>1$ and $s_B< \infty$, respectively.
\\ The assumption \eqref{sb} also ensures 
 that for every $k >0$ there exists  a constant $c$, depending only on $k$ and $s_b$, such that
\begin{equation}\label{bdelta2}
b(k t) \leq c b(t) \qquad \hbox{for $t \geq 0$.}
\end{equation}
Similarly, the fact that $s_B< \infty$ ensures that for every $k>0$ there exists  a constant $c$, depending only on $k$ and $s_B$, such that
\begin{equation}\label{Bdelta2}
B(k t) \leq c B(t) \qquad \hbox{for $t \geq 0$.}
\end{equation}
The Young conjugate $\widetilde{B}$ of $B$ is the Young function defined 
 as
\begin{equation}
    \widetilde{B}(t)=\sup\{st-B(s):\,s\geq 0 \} \qquad \text{for $t \geq 0$}.
\end{equation}
Since $i_B>1$, a  property parallel to \eqref{Bdelta2} holds for $\widetilde B$. Namely,
for every $k >0$ there exists  a constant $c$, depending only on $k$ and $i_B$, such that
\begin{equation}\label{Btildedelta2}
\widetilde B(k t) \leq c \widetilde B(t) \qquad \hbox{for $t \geq 0$.}
\end{equation}
The property  $s_B< \infty$ also implies that there exists a constant $c$ such that
\begin{equation}\label{feb1}
\widetilde B(b(t)) \leq c B(t)  \qquad \hbox{for $t \geq 0$.}
\end{equation}
Let $\beta, \gamma, F:[0,\infty)\to [0,\infty)$ be the functions defined as
\begin{equation}\label{oct101}
   \beta(t)=b(t)\,t\,,\quad \gamma(t)=b^{-1}(t)\,t\,,   \quad F(t)=\int_0^t b^2(s)\,ds\,.
\end{equation}
Then,
\begin{equation}\label{b1}
    c_1\,B(t)\leq \beta(t)\leq c_2\,B(t)\,,\quad  c_1\,\widetilde{B}(t)\leq \gamma(t)\leq c_2\,\widetilde{B}(t)\,,
\end{equation}
 and
\begin{equation}\label{b2}
    F(t)\leq t\,b^2(t)\leq c_3\,F(t)\,.
\end{equation}
for some positive constants $c_1, c_2$ and $c_2$ depending on $i_a,s_a$ (see e.g. \cite[Proposition 2.15]{cia11}).
\\ Let  $a$ be defined as above and let
 $\e\in (0,1)$. Define the function  $a_\e(t): [0, \infty) \to [0, \infty)$ as
\begin{equation}\label{a_e}
    a_\e(t)=\frac{{a}(\sqrt{\e+t^2})+\e}{1+\e\,{a}(\sqrt{\e+t^2})} \quad \text{for $t\geq 0$,}
\end{equation}
the function $ b_\e(t): [0, \infty) \to [0, \infty)$ as
\begin{equation}\label{b_e}
    b_\e(t)=a_\e(t) t \quad \text{for $t\geq 0$,}
\end{equation}
and the Young function $B_\e$ as
\begin{equation}\label{B_e}
    B_\e(t)=\int_0^t b_\e(s)\,ds\,.
\end{equation}
 Then,  $ a_\e\in C^1([0,\infty))$,
\begin{equation}\label{appr:1}
  \e\leq a_\e(t)\leq \e^{-1}\quad\text{for }t\geq 0,
\end{equation}
and
\begin{equation}\label{appr:2}
    \min\{i_a,0\}\leq i_{a_\e}\leq s_{a_\e}\leq \max\{s_a,0\}.
\end{equation}
Moreover,   given any  $M>0$,
\begin{equation}\label{cortona10}
     \lim\limits_{\e\to 0^+}b_\e (t)=b(t) 
\end{equation}
uniformly in $[0,M]$-- see \cite[Proof of Lemma 4.5]{cia14}.

\subsection{The functions $H$ and $\mathcal A$}\label{sec:tutto}
Let $H$ be a norm in $\rn$  such that $H^2\in C^2(\rn\setminus \{0\})$ and fulfilling 
\eqref{ell:H}.  Then,
\begin{equation}\label{reg:H}
   H\in C^1(\R^n)\,,\quad\text{and}\quad H^2\in C^{1,1}(\R^n)\,.
\end{equation}
Moreover,
\begin{equation}\label{bounds:H}
    \l\,|\xi|^2\leq H^2(\xi)\leq \L\,|\xi|^2\quad  \text{for $\xi\in \rn$,}
\end{equation}
and
\begin{equation} \label{bound:nabH1}
\sqrt{\lambda} \leq |\nabla_\xi H(\xi)| \leq \sqrt{\Lambda}  \quad \text{for $\xi\neq 0$.}
\end{equation}
Here $\nabla_\xi H$ denotes the differential of $H$ with respect to its original variable $\xi\in \R^n\setminus\{0\}$. Let 
$\A : \rn \to \rn$ be the vector field  defined as 
\begin{equation} \label{def_A}
    \A(\xi)=\nabla_\xi B\big( H(\nabla u)\big)=\begin{cases} b\big(H(\xi)\big) \nabla_\xi H(\xi) &\quad \text{if $\xi \neq 0$}
\\ 0  &\quad \text{if $\xi = 0$.}
\end{cases}
\end{equation}
The inequality 
\eqref{bound:nabH1} yields
\begin{equation}\label{A:bds_sopra}
\sqrt{\l}\, b(H(\xi)) \leq |\A(\xi)|\leq \sqrt{\L}\, b(H(\xi)) \quad \text{for $\xi \in \rn$.}
\end{equation}       
Notice that    the function $\mathcal A$ admits the alternate expression 
\begin{equation} \label{def_Abis}
    \A(\xi) = a\big(H(\xi)\big)\tfrac{1}{2}\nabla_\xi H^2(\xi) \quad \text{for $\xi \neq 0$.}
\end{equation}
Thanks to \eqref{Bdelta2},
%
%
%
%
the inequalities  \eqref{bounds:H} imply that
\begin{equation}\label{B:equiv}
     c  B(|\xi|)  \leq B(H(\xi))\leq C  B(|\xi|) \quad \text{for $\xi \in \rn$,}
\end{equation}
for suitable constants $c=c(s_B, \l,\L)$ and $C=C(s_B,\l,\L)$,
and
\begin{equation}\label{b:equiv}
     c  b(|\xi|)  \leq b(H(\xi))\leq   C b(|\xi|) \quad \text{for $\xi \in \rn$,}
\end{equation}
for suitable constants $c=c(s_b, \l,\L)$ and $C=C(s_b,\l,\L)$. 
\\
From the homogeneity of $H^2$, we deduce that
\begin{equation*} 
        \A(\xi)\cdot\xi = a(H(\xi)) H^2(\xi) = b(H(\xi))H(\xi) \quad \text{for $\xi \in \rn$.}
\end{equation*}
Coupling the latter equation with the inequality \eqref{Bb} yields
\begin{equation}\label{A:bds}
        \A(\xi)\cdot\xi  \geq B\big(H(\xi)\big) \quad \text{for $\xi \in \rn$.}
\end{equation}
One has that
\begin{equation}\label{monot:A}
    \big(\A(\xi)-\A(\eta)\big)\cdot(\xi-\eta)>0 \quad \text{for  $\xi, \eta \in \rn$, with $\xi\neq \eta$.}
\end{equation}
Observe that $\nabla_\xi \A(\xi)$ is the symmetric matrix given by
$$\nabla_\xi \A(\xi)=
\left(\frac{\partial \A^i(\xi)}{\partial\xi_j}\right)_{i,j=1,\ldots,n}  \quad \text{for $\xi \neq 0$.}
$$
Moreover,
\begin{equation}\label{mon:updown}
\l\,\min\{1,i_b\}\,a\big(H(\xi)\big)\,|\eta|^2 \leq \nabla_\xi\A(\xi)\, \eta \cdot \eta
 \leq \L\,\max\{1,s_b\}\,a\big(H(\xi)\big)\,|\eta|^2\,,
\end{equation}
for $\xi\neq 0$ and $\eta\in \rn$ -- see \cite[Lemma 3.2]{accfm}.
\\
Let $\e\in (0,1)$ and let $b_\e$ be function given by \eqref{b_e}. We define
 the vector field $\A_\e: \rn \to \rn$ as
\begin{equation}\label{def:Ae'}
    \A_\e(\xi)=\nabla_\xi B_\e\big( H(\xi)\big)=\begin{cases} b_\e\big(H(\xi)\big)\nabla_\xi H(\xi) & \quad \text{if $\xi \neq0$}
\\ 0 & \quad \text{if $\xi =0$.}
\end{cases}
\end{equation}
Notice the alternative formula
\begin{equation}\label{def:Ae}
    \A_\e(\xi) = a_\e\big(H(\xi)\big)\tfrac{1}{2}\nabla_\xi H^2(\xi) \quad \text{if $\xi \neq 0$.}
\end{equation}
Given any $M>0$, one has that
\begin{equation}\label{appr:3}
    \lim\limits_{\e\to 0^+}\A_\e(\xi)=\A(\xi) 
\end{equation}
uniformly in $\{\xi\in\rn\,:\,|\xi|\leq M\}$.
\\
Analogously to \eqref{mon:updown}, 
\begin{equation}\label{pr:ae}
    \l\,\min\{1,i_b\}\, a_\e\big(H(\xi)\big)\,|\eta|^2 \leq \nabla_\xi \A_\e(\xi)\, \eta \cdot \eta
\leq \L\,\max\{1,s_b\}\, a_\e\big(H(\xi)\big)\,|\eta|^2
\end{equation}
for $\xi\neq 0$ and $\eta\in \rn$. 
In particular,  
\begin{equation}\label{est:Ae}
   \e\,\l\,\min\{1,i_{b}\}\,\mathrm{Id} \leq \nabla_\xi \A_\e(\xi)
\leq\e^{-1}\,\L\,\max\{1,s_{b}\}\,\mathrm{Id} \quad \text{for $\xi \neq 0$.}
\end{equation}
Hence, the function $\A_\e$ is Lipschitz continuous.

\subsection{Function spaces}
Let $(\mathcal{R},m)$ be a  $\sigma$-finite non-atomic measure space. We denote by $\mathcal M(\mathcal{R})$ the space of real-valued measurable functions on $\mathcal{R}$.\\
Given
a Young function $B$, the Orlicz space $L^B(\mathcal{R})$ is the Banach space of functions $\phi \in \mathcal M(\mathcal{R})$ whose Luxemburg norm
\begin{equation*}
    \|\phi\|_{L^B(\mathcal{R})}=\inf\Big\{t>0\,:\:\int_\mathcal{R} B\Big( \frac{|\phi|}{t}\Big)\,dm\leq 1\Big\}
\end{equation*}
is finite. 
The classical Lebesgue space $L^p(\mathcal{R})$, with $p\geq 1$, is reproduced with the choice $B(t)=t^p$.
\\ The Lorentz spaces can be defined via rearrangements of functions.
Let $\phi \in \mathcal M (\mathcal{R})$.
The distribution function $\mu_\phi\:[0,\infty)\to [0,m(\mathcal{R})]$ of a  $\phi$ 
is given by
\begin{equation}
    \mu_\phi(t)=m\big(|\big\{ x\in \mathcal{R}\,:\,|\phi(x)|>t\big\}\big) \quad \text{for $t \geq 0$.}
\end{equation}
The decreasing rearrangement $\phi^*:[0,\infty)\to [0,\infty]$ of $\phi$ 
is defined as
\begin{equation}
    \phi^*(s)=\inf\big\{t\geq 0\,:\,\mu_\phi(t)\leq s \big\} \quad \text{for $s\geq 0$,}
\end{equation}
and it is the unique non-increasing right-continuous function on $[0,\infty)$ which is equidistributed with $\phi$. 
\\ The Hardy-Littlewood inequality \cite[Section 7.2]{pick} tells that
\begin{equation}\label{HL}
    \int_{\mathcal{R}}|\phi(x)\,\psi(x)|\,dm\leq \int_0^\infty \phi^*(s)\,\psi^*(s)\,ds
\end{equation}
for every $\phi, \psi \in \mathcal M(\mathcal{R})$.
\\
The maximal function $\phi^{**}: (0, \infty) \to [0,\infty]$ associated with $\phi^*$ is given by 
\begin{equation}
    \phi^{**}(t)=\frac{1}{t}\int_0^t \phi^*(s)\,ds\,\quad \text{for $t>0$.}
\end{equation}
One has that
\begin{equation}\label{sub:add}
     (\phi+\psi)^{**}\leq  \phi^{**}+\psi^{**}.
\end{equation}
  Moreover, 
\begin{align}
    \label{oct100}
\text{if \,\,$\phi^{**}\leq \psi^{**}$, \,\, then \,\, $\|\phi\|_{L^p(\mathcal R)}\leq \|\psi\|_{L^p(\mathcal R)}$,}
\end{align}
for $\phi, \psi \in \mathcal M(\mathcal{R})$ and $p\in [1, \infty]$.
\\ A property of rearrangements in connection with convolutions guarantees that, if $\mathcal R = \rn$ equipped with the Lebesgue measure and $\rho_k$ is a sequence of standard radially decreasing mollifiers, then
\begin{equation}\label{convol:star1}
  (\phi\ast\rho_k)^{**}(r)\leq c(n)\,\phi^{**}(r)\quad\text{for $r>0$,}
\end{equation}
for $\phi\in L^1_{loc}(\rn)$ -- see e.g. \cite[Lemma 1.8.7]{ziemer}.
\\
For $q\in [1,\infty)$ and $\sigma\in (0,\infty)$, the Lorentz spaces $L^{q,\sigma}(\mathcal{R})$ and $L^{(q,\sigma)}(\mathcal{R})$ are the sets of measurable functions $\phi$ such that the quantities
\begin{equation}\label{lore:norm}
    \|\phi\|_{L^{q,\sigma}(\mathcal{R})}=|| s^{\frac{1}{q}-\frac{1}{\sigma}}\phi^*(s)\|_{L^\sigma(0,\infty)}\,,
\end{equation}
and 
\begin{equation}\label{lore:norm1}
    \|\phi\|_{L^{(q,\sigma)}(\mathcal{R})}=|| s^{\frac{1}{q}-\frac{1}{\sigma}}\phi^{**}(s)\|_{L^\sigma(0,\infty)}\,,
\end{equation}
are respectively finite.  If $\sigma\in [1,\infty)$, then $L^{q,\sigma}(\mathcal{R})$ is a Banach space. Moreover, if $q\in (1,\infty)$, we have 
\begin{equation}\label{lor:equiv}
 \|\phi\|_{L^{q,\sigma}(\mathcal{R})}\leq \|\phi \|_{L^{(q,\sigma)}(\mathcal R)}\leq c(q,\sigma)\,\|\phi\|_{L^{q,\sigma}(\mathcal{R})}.
\end{equation}
Hence, replacing $\phi^*$ with $\phi^{**}$ in \eqref{lore:norm} results in an equivalent quantity, up to multiplicative constants independent of $\phi$-- see \cite[Proposition 8.1.8]{pick}.
\\
On the other hand, if $m(\mathcal{R})<\infty$, then 
\begin{equation}
    L^{(1,1)}(\mathcal{R})=L\,\log L(\mathcal{R})\,,
\end{equation}
up to equivalent norms, where $L\,\log L(\mathcal{R})$ denotes the  Orlicz space  associated with the Young function $A(t)=t\,\log_+ t$, see \cite[Chapter 4, Lemma 6.2]{bennett} or \cite[Remark 8.1.10]{pick}. We also set
\begin{equation*}
    (L^2)^{(1,\frac{1}{2})}(\mathcal{R})=
    \{\phi\in \mathcal M(\mathcal{R}): \phi^2 \in L^{(1, \frac 12)}(\mathcal{R})\}
    %
    %
\end{equation*}
and define   the functional
\begin{equation*}
    \|\phi\|_{(L^2)^{(1,\frac{1}{2})}(\mathcal{R})}
    =\|\phi^2\|^{1/2}_{L^{(1,\frac{1}{2})}(\mathcal{R})}
\end{equation*}
for $\phi \in \mathcal M (\mathcal R)$.
\\
Assume now that $\Omega$ is a open set in $\rn$ equipped with the Lebesgue measure, and 
 $k,m\in \N$. Then the Lorentz-Sobolev space $W^mL^{q,\sigma}(\Omega)$ is defined as
\begin{align}\label{LS}
    W^mL^{q,\sigma}(\Omega)=\Big\{ u\in L^{q,\sigma}(\Omega):&\, \text{$v$ is $m$-times  weakly differentiable }
    \\  
   & \text{and $|\nabla^j u|\in L^{q,\sigma}(\Omega)$ for $j=1,\dots,m$}\Big\}\,.\nonumber
\end{align}
Here $\nabla^j u$ denotes the $j$-th order distributional gradient of $u$.
\\
Similarly, given a Young function $B$,
the Orlicz-Sobolev space $W^{1,B}(\Omega)$ is defined as the Banach space 
\begin{equation}\label{OS}
    W^{1,B}(\Omega)=\Big\{u\in L^B(\Omega)\,:\,\nabla u\in L^B(\Omega)\Big\}\,,
\end{equation}
endowed with norm $\|u\|_{W^{1,B}(\Omega)}=\|u\|_{L^{B}(\Omega)}+\|\nabla u\|_{L^B(\Omega)}$.  We also define its subspaces
\begin{equation}
     W^{1,B}_0(\Omega)=\Big\{u\in W^{1,B}(\Omega)\,:\,\text{the continuation of $u$ by 0 outside $\Omega$ belongs to $W^{1,B}(\rn)$}\Big\}\,,
\end{equation}
and
\begin{equation}
    W^{1,B}_\perp (\Omega) =\Big\{u\in W^{1,B}(\Omega)\,:\,\int_\Omega u\,dx=0\Big\}\,.
\end{equation}

\subsection{Classes of domains} 


An open set $\Omega$ in $\rn$ is called a Lipschitz domain  if 
 there exist  constants $L_\Om>0$ and $R_\Om \in (0, 1)$ 
such that, for every $x_0\in \partial \Om$ and $R\in (0, R_\Om]$ there exist an orthogonal coordinate system centered at $0\in\rn$ and an $L_\Om$-Lipschitz continuous function 
$\varphi : B'_{R}\to (-\ell, \ell)$, where $B'_{R}$ denotes the ball in $\mathbb R^{n-1}$, centered at $0'\in \R^{n-1}$ and with radius $R$, and
\begin{equation}\label{ell}
\ell = R (1+L_\Om),
\end{equation}
satisfying
\begin{align}\label{may100}
    &\partial \Om \cap \big(B'_{R}\times (-\ell,\ell)\big)=\{(x', \varphi (x'))\,:\,x'\in B'_{R}\},
    \\
    & \Om \cap \big(B'_{R}\times (-\ell,\ell)\big)=\{(x',x_n)\,:\,x'\in B'_{R}\,,\,\varphi (x')<x_n<\ell\}.\nonumber
\end{align}
We set
\begin{equation}\label{may101}
\mathfrak L_\Om = (L_\Om, R_\Om),
\end{equation}
and call $\mathfrak L_\Om$ a Lipschitz characteristic of $\Om$. 
\\
 Note that, in general, a  Lipschitz characteristic $\mathfrak{L}_\Om = (L_\Om, R_\Om)$ is not   uniquely determined. For instance, if $\partial \Omega \in C^1$, then  $L_\Om$ may be taken arbitrarily small, provided that  $R_\Om$ is chosen sufficiently small. 
\\
Let $\Omega$ be a Lipschitz domain as above and let $X(B_R')$ be a function space on $B_R'$. 
We say that $\Omega$ is of class $W^2X$,  and we write $\partial \Omega\in W^2X$, if the function $\varphi$ also belongs to the Sobolev type space $W^2X(B'_R)$. In analogy with \eqref{LS}
and \eqref{OS}, this means that $\varphi$ is a twice weakly differentiable function which belongs, together with its first and second-order weak derivatives, to $X(B'_R)$.

\subsection{Geometric and functional inequalities}

Assume that $\Omega$ is a bounded Lipschitz domain. Then, there exists a constant $c=c(n, d_\Omega, \mathcal{L}_\Omega)$ such that
\begin{equation}\label{rel:isoper1}
      \H^{n-1}\big(\partial^e E\cap \partial \Omega)\leq c\,\H^{n-1}(\partial^e E\cap \Omega),
  \end{equation}
for every set $E\subset \Omega$ of finite perimeter such that $|E|\leq |\Omega|/2$. Here, $\partial^e E$ denotes the essential boundary of $E$ in the geometric measure theoretical sense.
The inequality \eqref{rel:isoper1} follows, for instance, via \cite[Remark 5.10.2]{ziemer} and a covering argument for $\Omega$
by cylinders as in Equation \eqref{may100} of the definition of Lipschitz domain.
\\ The inequality \eqref{rel:isoper1} implies the trace inequality 
\begin{equation}\label{trace:leoni}
    \|\mathrm{Tr}(u)\|_{L^1(\partial \Omega)}\leq c(n,d_\Omega,\mathcal{L}_\Omega)\,\|u\|_{BV(\Omega)}\,,\quad u\in BV(\Omega),
\end{equation}
for some constant $c=c(n, d_\Omega, \mathcal{L}_\Omega)$ and every $u\in BV(\Omega)$. 
Here, $\mathrm{Tr}$ denotes the trace operator on $\partial \Omega$.
 \\ The relative isoperimetric inequality on $\Omega$ can also be deduced from  \eqref{rel:isoper1} and 
 tells us that
\begin{equation}\label{rel:isoper}
    |E|^{\frac{n-1}{n}}\leq c\,\H^{n-1}\big(\partial^e E\cap \Omega\big)\,,
\end{equation}
for  some constant $c=c(n, d_\Omega, \mathcal{L}_\Omega)$
and every set $E\subset \Omega$ of finite perimeter such that $|E|\leq |\Omega|/2$.
\\ The inequality \eqref{rel:isoper} in turn yields the Sobolev-Poincar\'e inequality 
\begin{equation}\label{poincare}
        \|u-u_\Omega\|_{L^{n'}(\Omega)}\leq c\,\|\nabla u\|(\Omega)\,,
    \end{equation}
    for every $u\in BV(\Omega)$, where again $c=c(n, d_\Omega, \mathcal{L}_\Omega)$.
\\
 Now, let $\Omega$ be any open set in $\rn$. Assume that $u\in W^{1,1}(\Omega)$ and $g:\Omega \to [0,\infty]$ is a  Borel function. Then
 the  coarea formula for Sobolev functions asserts that
\begin{equation}\label{coarea}
    \int_{\Omega}g\,|\nabla u|\,dx=\int_{\R}\int_{\{u=s\}}g\,d\H^{n-1}\,ds\,,
\end{equation}
 provided that a suitable precise representative of the function $u$ is chosen
 \cite[Proposition 2.1]{br:ziemer}.
 In particular
\begin{equation}\label{coarea:1}
    \int_{\{u>t\}}g\,|\nabla u|\,dx=\int_t^\infty\int_{\{v=s\}}g\,d\H^{n-1}\,ds\quad \text{for $t \in \mathbb R$}.
\end{equation}
The choice $g=\chi_{\{|\nabla v|=0\}}$ yields: 
\begin{equation}\label{t:hn1zero}
\H^{n-1}\big(\{|\nabla v|=0\}\cap \{v=t \}\big)=0\quad\text{for a.e. $t\in \R$.}  
\end{equation}
Thus, setting $g= \chi_{\{v=t\}}/|\nabla v|$ for all values 
 of $t$ which make \eqref{t:hn1zero} true, shows that
\begin{equation}\label{t:lebzero}
    |\{ v=t\}|=0\quad\text{for a.e. $t\in \R$.}
\end{equation}
If $\Omega$ is a bounded Lipschitz domain, then the coarea formula \eqref{coarea} and the 
 relative isoperimetric inequality 
\eqref{rel:isoper} enable one to deduce that
\begin{equation}\label{dis:distrib}
        1\leq c\,(-\mu'_u(t))^{1/2}\,\mu_u(t)^{-1/n'}\,\Big(-\frac{d}{dt}\int_{\{u>t\}}|\nabla u|^2 dx \Big)^{1/2}\quad \text{for a.e. $t\geq u^*(|\Omega|/2)$,}
    \end{equation}
    for some constant $c=c(n,d_\Omega,\mathcal{L}_\Omega)$ -- see \cite{Mazya69}.

\section{Boundary value problems with Orlicz growth}

As mentioned above, handling the convection term depending on $\nabla u$ in the equation \eqref{equation} requires a change of  the unknown variable $u$. This reduces our task to dealing with an equation with a right-hand side only depending on $x$. The  analysis of the regularity of solutions to the boundary value problems associated with the new equation yet relies upon approximation arguments at various levels. In particular, approximating the differential operator leads us to deal with elliptic equations whose growth is not anymore of plain power type. We begin with results parallel to those of Theorems \ref{thm:example1} and \ref{thm:example} for boundary value problems for anisotropic elliptic equations involving Orlicz type nonlinearities. They are general enough to be applied in the proof of \ref{thm:example1} and \ref{thm:example} outlined above, and have an independent interest. In particular, they extend results from  \cite{cia11}, where isotropic operators associated with the Euclidean norm $H(\xi)=|\xi|$ where considered.

The problems in question 
 arise as Euler-Lagrange equations of functionals of the form
\begin{equation}\label{funct_H}
J_H(u)= \int_\Omega B(H(\nabla u)) \,dx - \int_\Omega f u  \, dx\,,
\end{equation} 
where $B$ is a Young function \eqref{B} satisfying \eqref{ib}--\eqref{sb}.
The relevant Euler-Lagrange equation reads:
\begin{equation} \label{eq_anisotr_intro}
    -\mathrm{div}\big(\A(\nabla u)\big)= f \quad\text{in } \Omega \,,
\end{equation} 
where the function $\A : \rn \to \rn$ is defined by \eqref{def_A}. Clearly, the differential operator in \eqref{eq_anisotr_intro} reproduces \eqref{def:anpl}
with the choice $B(t)= \frac 1p t^p$ for  $p>1$.
\\
The Dirichlet problem for the equation \eqref{eq_anisotr_intro} takes the form \begin{equation}\label{eq:dir2}
\begin{cases}
        -\mathrm{div}\big(\A(\nabla u) \big)=f &\quad\text{in } \Omega
        \\
        u=0 &\quad\text{on }\partial \Omega.
\end{cases}
\end{equation}
Recall that a weak solution to this problem is a function $u\in W^{1,B}_0(\Omega)$ such that
\begin{equation}\label{weak:dir}
    \int_\Omega\A(\nabla u)\cdot \nabla \vphi\,dx=\int_\Omega f\,\vphi\,dx
\end{equation}
for every  $\vphi\in W^{1,B}_0(\Omega)$.
\\ The (co-normal) Neumann
problem for \eqref{eq_anisotr_intro} reads
\begin{equation}\label{eq:neu2}
\begin{cases}
        -\mathrm{div}\big(\A(\nabla u) \big)=f &\quad\text{in } \Omega
        \\
        \A(\nabla u) \cdot \nu =0 &\quad\text{on }\partial \Omega.
\end{cases}
\end{equation}
Under  the compatibility condition
\begin{equation}\label{mean:f0}
    \int_\Omega f\,dx=0,
\end{equation}
a weak solution to the 
 problem \eqref{eq:neu2} 
 is a function $u\in W^{1,B}(\Omega)$ such that
\begin{equation}\label{weak:neu}
    \int_\Omega\A(\nabla u)\cdot \nabla \vphi\,dx=\int_\Omega f\,\vphi\,dx
\end{equation}
for every $\vphi\in W^{1,B}(\Omega)$.
\\ Of course, the above definition require that the right-hand sides of the equations \eqref{weak:dir}
and \eqref{weak:neu}
be well defined for every trial function $\varphi$. A sharp integrability condition on $f$, depending on $B$ and $n$ through a sharp Sobolev type embedding in Orlicz spaces is available-- see \cite{ACCZ}. In our discussion, we assume that $f\in X_n(\Omega)$. As the latter space is contained in $L^n(\Omega)$ and $ W^{1,B}(\Omega)\subset W^{1,1}(\Omega)$ for every Young  function $B$,  the integrals involving $f$ in \eqref{weak:dir}
and \eqref{weak:neu} are certainly convergent, thanks to H\"older's inequality and the Sobolev inequality in $W^{1,1}(\Omega)$.

In the following statement we collect an existence and uniqueness result, and  a basic energy estimate, for weak solutions to the problems \eqref{eq:dir2} and \eqref{eq:neu2} with $f \in L^n(\Omega)$.

\begin{theorem}\label{existence}
    Let $\Omega$ be a bounded open set  in $\R^n$ and let $f\in L^n(\Omega)$. Assume that $H^2\in C^2(\R^n\setminus\{0\})$ and fulfills \eqref{ell:H}, and that $a\in C^1(0,\infty)$ and satisfies \eqref{c:a}.
    \\
    (i) There exists a unique weak solution $u\in W^{1,B}_0(\Omega)$ to the Dirichlet problem \eqref{eq:dir2}. Moreover, 
\begin{equation}\label{stima:energiadir}
        \int_\Omega B\big( H(\nabla u)\big)\,dx\leq c\,\widetilde{B}(\|f\|_{L^n(\Omega)})
    \end{equation}
    for some constant $c=c(n,\l,\L, i_a,s_a,|\Omega|)$.
    \\
    (ii) Assume, in addition, that $\Omega$ is a Lipschitz domain and $f$ satisfies the compatibility condition \eqref{mean:f0}. Then there exists a unique weak solution $u\in W^{1,B}_\perp(\Omega)$ to the Neumann problem \eqref{eq:neu2}. Moreover, 
\begin{equation}\label{stima:energianeu}
        \int_\Omega B\big( H(\nabla u)\big)\,dx\leq c\,\widetilde{B}(\|f\|_{L^n(\Omega)})
    \end{equation}
    for some constant $c=c(n,\l,\L, i_a,s_a, d_\Omega, \mathcal L_\Om)$.
\end{theorem}
\begin{proof} The proof of the existence and uniqueness of solutions follows via standard variational arguments, analogous to those  
   employed in \cite[Theorem 2.13]{cia11} in the case when $H(\xi)=|\xi|$. The details are omitted for brevity. 
Choosing the test function $\varphi=u$ in  equation \eqref{weak:dir} or \eqref{weak:neu} and using the  property \eqref{A:bds} of $\A$ enable us to deduce that
    \begin{equation}\label{fist:energy}
        \int_\Omega B\big( H(\nabla u)\big)\,dx\leq \int_\Omega \A(\nabla u)\cdot\nabla u\,dx=\int_\Omega f\,u\,dx\,.
    \end{equation}    
    On the other hand, by  H\"older's, a Sobolev inequality,  and \eqref{bounds:H}, one has that
    \begin{equation*}
        \int_\Omega f\,u\,dx\leq  \|f\|_{L^n(\Omega)}\,\|u\|_{L^{n'}(\Omega)}\leq c_0\,\|f\|_{L^n(\Omega)}\,\int_\Omega H(\nabla u)\,dx.
    \end{equation*}  
Note that one has to use the Sobolev 
inequality
for functions vanishing on $\partial \Omega$    for the problem \eqref{eq:dir2}, and the Poincar\'e-Sobolev inequality \eqref{poincare} for the problem \eqref{eq:neu2}. Accordingly,
     the constant  $c_0$  depends on $n,\l,\L$ for the former problem,  and
     on $n,\l,\L,\mathcal{L}_\Omega,d_\Omega$ for the latter.     
Thanks
  to Young's inequality, the properties \eqref{usual:b} and \eqref{dic105},  and Jensen's inequality,  
    \begin{align*}
\|f\|_{L^n(\Omega)}\,\int_\Omega H(\nabla u)\,dx & \leq c_\varepsilon\,\widetilde{B}(\|f\|_{L^n(\Omega)})+c_1\,\varepsilon\, B\Big(\int_\Omega H(\nabla u)\,dx\Big)
        \\
        &\leq c_\varepsilon\,\widetilde{B}(\|f\|_{L^n(\Omega)})+c_2\,\varepsilon \int_\Omega B\big(H(\nabla u)\big)\,dx
    \end{align*}
    for some constants $c_\varepsilon=c(\varepsilon,i_a,s_a)$, $c_1= (c_0,i_a,s_a)$, and $c_2(c_1,|\Omega|)$. Altogether, the inequalities \eqref{stima:energiadir} and \eqref{stima:energianeu}
follow with 
 with the choice of $\varepsilon=c_2/2$.   
\end{proof}

Our results on the global boundedness of the gradient of solutions to the problems 
\eqref{eq:dir2} and \eqref{eq:neu2} under minimally regular domains are stated in the following theorem. As pointed out above, unlike the case when the right-hand side may depend on $\nabla u$, the a priori assumption on the boundedness of $u$ is not needed.

\begin{theorem}[Minimally integrable curvatures]\label{thm:dirneu}
 Let $\Omega$ be a bounded Lipschitz domain in $\rn$ such that $\partial \Omega \in W^2L^{(n-1,1)}$. Let  $H$   be a norm in $\rn$ such that $H^2\in C^2(\R^n\setminus\{0\})$ and fulfilling \eqref{ell:H}, and let $a\in C^1(0,\infty)$ be a function satisfying \eqref{c:a}. Assume that
    either
    $n\geq 3$ and $f\in  L^{n,1}(\Omega)$, or $n=2$ and $f \in (L^2)^{(1,\frac{1}{2})}(\Omega)$.  
    Let $u$ be the weak solution to the Dirichlet  problem \eqref{eq:dir2} or the Neumann problem \eqref{eq:neu2}.
Then $u\in W^{1,\infty}(\Omega)$, and there exists a constant $c=c(n,i_b,s_b,\l,\L,\Omega)$ such that
\begin{equation}\label{stima:dirichlet}
    \|\nabla u\|_{L^\infty(\Omega)}\leq c\,b^{-1}\big(\|f\|_{X_n(\Omega)}\big),
\end{equation}
where $X_n(\Omega)$ is defined by \eqref{X:f}.
\end{theorem}

The same conclusion as in 
Theorem \ref{thm:dirneu}
 holds if $\Omega$ is any bounded convex domain. In this case, the  constant $c$ in the bound for $ \|\nabla u\|_{L^\infty(\Omega)}$  depends on $\Omega$ only through  its Lipschitz characteristic $\mathcal{L}_\Omega$, its  Lebesgue measure $|\Omega|$, and its diameter $d_\Omega$. 

\begin{theorem}[Convex domains]\label{thm:conv}
        Under the same assumptions on $H$, $a$ and $f$ as in Theorem \ref{thm:dirneu}, suppose that $\Omega$ is a convex domain of $\rn$. Let $u$ be a weak solution to the Dirichlet problem \eqref{eq:dir2} or the Neumann problem \eqref{eq:neu2}. Then $u\in W^{1,\infty}(\Omega)$, and there exists a constant $c=c(n,i_b,s_b,\l,\L,\mathcal{L}_\Omega, |\Omega|, d_\Omega)$ such that
\begin{equation}\label{stima:convex}
    \|\nabla u\|_{L^\infty(\Omega)}\leq c\,b^{-1}\big(\|f\|_{X_n}\big)\,.
\end{equation}
\end{theorem}

A close inspection of the proof will reveal that the constant $c$ in \eqref{stima:convex} in fact  depends  on $|\Omega|$ only via a lower bound, and on $d_\Omega$ only via  an upper bound.

\section{Identities for vector fields}\label{sec:vector}
In this section, we establish some differential and integral identities for vector fields.
In what follows, the notation $\partial_j$ stands for $\partial_{x_j}$, the partial derivative with respect to the variable $x_j$.  Moreover,   we adopt the convention about summation over repeated indices.
\\ Let $\Omega$ be an open set in $\rn$. Given a
 vector field $V:\Omega\to \R^n$, with $V=(V^1,\dots,V^n)$, we denote by $\nabla V$ the matrix-valued function defined as
\begin{equation*}
    \nabla V=(\partial_j V^i)_{i,j=1}^n.
\end{equation*}
Therefore, $\nabla V \, V$ is the vector whose $i$-th component agrees with $V^j\partial _jV^i$. 
\begin{lemma}\label{lemma:cruc}
Assume  that   $V,\,W: \Om \to \rn$ with   $V,\,W\in C^2(\Omega)$.
Then
\begin{equation}\label{cruch1}
        \big(\mathrm{div}V\big)\,\big(\mathrm{div}W\big)=\mathrm{tr}\big(\nabla V\,\nabla W\big)+\mathrm{div}\big(V\,\big(\mathrm{div}W\big)-\nabla W\,V\big)\,.
    \end{equation}
\end{lemma}

\begin{proof}
    Computations show that
    \begin{equation*}
        \mathrm{div}\Big(V\,\big(\mathrm{div}W\big) \Big)=\big(\mathrm{div}V\big)\,\big(\mathrm{div}W\big)+V\cdot \nabla\big(\mathrm{div}W\big)\,.
    \end{equation*}
   Schwarz's theorem about second mixed derivatives and  an exchange of  the indices tell us that
    \begin{align*}
\mathrm{div}\big(\nabla W\,V\big)&=\partial_j\big( \partial_i W^j\,V^i\big)=\partial_j V^i\,\partial_i W^j+V^i\,\partial_i\partial_j W^j
    \\
    &= \mathrm{tr}\big(\nabla V\,\nabla W\big)+V\cdot \nabla\big(\mathrm{div}W\big)\,.
    \end{align*}
Equation \eqref{cruch1} then follows by subtracting the two equations above.
\end{proof}

Let $\Omega$ be a set of locally finire perimeter. We denote by $\partial^* \Omega$ its reduced boundary, and denote by $\nu=\nu(x)$ its normal at $x\in \partial^*\Omega$. 
\\
Given a vector field $V: \rn\to \rn$, its tangential component $V_T$ on $ \partial^* \Omega$ is defined by 
\begin{equation*}
    V_T=V-(V\cdot \nu)\,\nu\,,\quad\text{on $\partial^*\Omega$.}
\end{equation*}
The notations $\nabla_T$ and $\mathrm{div}_T$ are adopted for the tangential gradient and divergence operators on $\partial^*\Omega$. 
Therefore, if $u\in C^1(\rn)$, then
\begin{equation*}
    \nabla_T u=\nabla u- (\partial_\nu u)\,\nu \quad \text{on $\partial^*  \Omega$,}
\end{equation*}
where $\partial_\nu u=\nabla u\cdot \nu$ denotes the normal derivative of $u$. Moreover, for $V\in C^1(\rn)$  we have
\begin{equation*}
    \mathrm{div}_T V=\mathrm{div}V-(\partial_\nu V\cdot\nu) \quad  \text{on $\partial^* \Omega$.}
\end{equation*}

\begin{lemma}\label{quick:form-new}
\rm{
 Let $\Omega\subset \rn$ be a bounded open set of finite perimeter. Assume that  $V,\,W: \rn \to \rn$ are vector fields of class $C^1$ in an open neighborhood $\mathcal{U}$ of $\overline{\Omega}$.
Then, 
\begin{align}\label{formuletta1}
\int_\Omega\big(\mathrm{div}\,V\big)\,\big(\mathrm{div}\,W\big) \,dx=&\int_\Omega\mathrm{tr}\big(\nabla\,V\,\nabla W\big)\,dx+\int_{\partial^*\Omega}\Big((\mathrm{div} W)\,V\cdot\nu- \nabla W\, V \cdot \nu  \Big)\,d\mathcal{H}^{n-1}\,,
\end{align}
and
\begin{equation}\label{gen21}
    (\mathrm{div}\,W)\,V\cdot \nu-  \nabla W\, V \cdot\nu =(\mathrm{div}_T W)\,V\cdot\nu-\nabla_T W \, V_T \cdot \nu \quad \text{on $\partial^* \Omega$}\,.
\end{equation}
}
\end{lemma}

\begin{proof}
We may assume that $V,W$ are compactly supported in $\rn$, otherwise it suffices to multiply them by a smooth compactly supported function in $\mathcal{U}$, which  equals to 1 in  $\overline \Omega$. 
Thus,  the vector fields $V,\,W$ can be approximated via standard convolutions, by  sequences $\{V_k\}$ and $\{W_k\}$ of smooth, compactly supported functions in $\rn$, such that
\begin{equation}\label{gen17}
V_k \to V\,, \quad W_k \to W\quad \text{in $C^1(\rn)$.}
\end{equation}
Fix $k \in \mathbb N$. Thanks to the identity \eqref{cruch1} and the divergence theorem on sets of finite perimeter, we have 
\begin{align}\label{formuletta11g}
\int_\Omega\big(\mathrm{div}\,V_k\big)\,\big(\mathrm{div}\,W_k\big) \,dx=&\int_\Omega\mathrm{tr}\big((\nabla\,V_k)\,(\nabla\,W_k)\big)\,dx+\int_{\partial^*\Omega}\Big((\mathrm{div}W_k)\, V_k-  \nabla W_k\, V_k  \Big)\cdot \nu  \,d\mathcal{H}^{n-1}\,.
    \end{align}
Passing to the limit as $k \to \infty$ in the latter equation   yields \eqref{formuletta1}.
\\
Next, by subtracting  the equations
\begin{equation*}
    (\mathrm{div}\,W_k)\,V_k\cdot\nu=(\mathrm{div}_T\,W_k)\,V_k\cdot\nu+(\partial_\nu W_k\cdot \nu)\,V_k\cdot \nu
\end{equation*}
and
\begin{equation*}
    \nabla W_k\, V_k \cdot\nu = \nabla_T W_k \, (V_k)_T \cdot \nu+ (\partial_\nu W_k\cdot \nu )\,(V_k\cdot\nu)\,,
\end{equation*}
 which hold on $\partial^* \Omega$, we obtain
\begin{equation}
    (\mathrm{div}\,W_k)\,V_k\cdot \nu-  \nabla W_k\, V_k \cdot\nu =(\mathrm{div}_T W_k)\,V_k\cdot\nu-\nabla_T W_k \, (V_k)_T \cdot \nu \quad\text{on }\partial^* \Omega\,.
\end{equation}
Equation \eqref{gen21} follows by 
letting $k\to \infty$.
\end{proof}

In our proofs, a  role is played by the \textit{second fundamental form} and its anisotropic counterpart. Let $\Omega$ be a  bounded open set with $\partial \Omega \in C^2$. The classical second fundamental form $\B$ of $\partial \Omega$ agrees with the shape operator (also called Weingarten operator) which is given by $\nabla_T \nu$.  Here, and in what follows, $C^1$--functions and vector fields are assumed to be extended, with the same regularity, in a neighborhood of $\partial \Omega$. 
In particular, at any point $x\in \partial \Omega$,  $\nabla_T \nu$ maps  the tangent space $\nu^\perp$ at $x$ onto itself, and we have
\begin{equation*}
    \B(\eta,\zeta)=\nabla_T \nu\, \eta\cdot \zeta\quad\text{for $\eta, \zeta\in \nu^\perp$.}
\end{equation*}
Given $V,W:\partial \Omega\to \R^n$ be $C^1$--vector field on $\partial \Omega$, and notice that by Leibniz rule for tangential derivatives, we have
\begin{align}
\label{gen24:neu}
     \nabla_T W\,V_T\cdot \nu&=\nabla_T(W\cdot \nu)\,V_T- \nabla_T\nu\,V_T\cdot W_T
    =\nabla_T(W\cdot \nu)\,V_T- \B(V_T,W_T)\,.  
\end{align}

We now define the anisotropic second fundamental form. For its properties recalled below, we  refer to \cite[Section 3]{accfm}.
By the homogeneity  of the norm $H$, we have that $\nabla_\xi^2 H(\nu)$ is a positive isomorphism of the tangent space $\nu^\perp$. Hence,
\begin{equation}\label{gen11}
\nabla_\xi^2 H(\nu)\,(\nabla_T \nu) : \nu^\perp \to  \nu^\perp.
\end{equation}
 The \textit{anisotropic second fundamental form of $\partial \Omega$}, denoted by $\mathcal{B}^H$, is accordingly defined as
\begin{equation*}
\B^H(\eta,\zeta)=\nabla_\xi^2 H(\nu)\nabla_T \nu\,\eta \cdot \zeta \quad \quad \text{for $\eta, \zeta\in \nu^\perp$.}
\end{equation*}
Thus,
\begin{equation} \label{def_BH}
\mathcal{B}^H=\nabla_\xi^2 H(\nu)\,\nabla_T \nu=\nabla_T \big(\nabla_\xi H(\nu)\big).
\end{equation} 
Furthermore,  the \textit{anisotropic mean curvature} is given by
\begin{equation} \label{def_AMC}
    \mathrm{tr}\,\mathcal{B}^H=\mathrm{div}_T\big(\nabla_\xi H(\nu)\big)=\mathrm{trace}\big(\nabla^2_\xi H(\nu)\,\nabla_T \nu  \big) \,.
\end{equation}
When $H$ is the Euclidean norm, one has that $\nabla_\xi^2 H(\nu)= {\rm Id}_{\nu^\perp}$, and hence $\B^H=\B$,  the classical second fundamental form of $\partial \Omega$.
\\
Notice that there exist positive constants $c=c(n,\l,\L)$ and $C=C(n,\l,\L)$  such that
\begin{equation}\label{equiv:BBH}
    c\,|\B|\leq|\B^H|\leq C\,|\B| \qquad \text{on $\partial \Omega$.}
\end{equation}
Furthermore,
\begin{equation}\label{feb315}
     \text{if \,\, $\B\geq 0\,\,[>0]$, \,\, then \,\, $\mathrm{tr} (\B^H)\geq 0\,\,[>0]$. }
 \end{equation}
 Finally, let us recall that, if the function $v$ is twice continuously differentiable in a neighborhood of $\partial \Omega$, and 
 $v=0$ on $\partial \Omega$, then
\begin{align}
\label{ag:24}
\mathrm{div}_T\Big(\tfrac{1}{2}\nabla_\xi H^2(\nabla v) \Big) \,\tfrac{1}{2}\nabla_\xi H^2(\nabla v)\cdot \nu
   &-\nabla_T\Big[\tfrac{1}{2}\nabla_\xi H^2(\nabla v)\Big] \,\Big[\tfrac{1}{2}\nabla_\xi H^2(\nabla v)\Big]_T\cdot \nu
    \\ 
    & =H(\nu)\,H^2(\nabla v)\,\mathrm{tr}\,\mathcal{B}^H\,\quad\text{on $\partial \Omega$,}\nonumber
\end{align}
see
\cite[Equation (4.28)]{accfm}.

\section{ Proof of Theorem \ref{thm:dirneu}, Dirichlet problems}\label{proof:dir}
Here we are concerned with global Lipschitz continuity of solutions to the Dirichlet problem \eqref{eq:dir2}.

\begin{proof}[Proof of  Theorem \ref{thm:dirneu},  Dirichlet problems] The proof requires approximation arguments at various levels. We split it into steps, for ease of presentation.
\par \noindent \textit{Step 1.} Here we assume that
\begin{equation}\label{fcinf}
    f\in C^\infty(\overline{\Omega})
\end{equation}
and
\begin{equation}\label{deom:inf}
    \partial \Omega \in C^\infty\,.
\end{equation}
Let $u$ be the solution to the problem \eqref{eq:dir2}. 
We begin by approximating $B(t)$ via a family of functions with quadratic growth $B_\e$ defined as in \eqref{B_e}, and the vector field $\mathcal A$ via $\mathcal A_\e$ defined by \eqref{def:Ae'}. Notice that $ W^{1,B_\e}_0(\Omega)=  W^{1,2}_0(\Omega)$, as the function $B_\e(t)$  is equivalent to $t^2$ for every $\e \in (0,1)$.
\\ Consider the family of  functions $\{u_\e\}$, where $u_\e\in W^{1,2}_0(\Omega)$ is the unique weak solution to the Dirichlet problem
\begin{equation}\label{eq:ue}
    \begin{cases}
            -\mathrm{div}\big(\A_\e(\nabla u_\e)\big)=f\quad&\text{in }\Omega
    \\
    u_\e=0\quad&\text{on }\partial\Omega\,.
    \end{cases}
\end{equation}
The existence and uniqueness of this solution is guaranteed, for instance, by Theorem \ref{existence}.
\\ We claim that
\begin{equation}\label{reg:ue}
    u_\e\in C^{1,\beta}(\overline{\Omega})\cap W^{2,2}(\Omega)\,,
\end{equation}
for some $\beta\in (0,1)$ independent on $\e$, and
\begin{equation}\label{conv:C1u}
    u_\e\to u\quad\text{in $C^1(\overline{\Omega})$.}
\end{equation}
Indeed, thanks to the properties \eqref{appr:2} and \eqref{pr:ae}, we can apply \cite[Theorem 2]{tal} and find a constant $c_0>0$ independent on $\e$ such that $\|u_\e\|_{L^\infty(\Omega)}\leq c_0$. It then follows from \cite[Theorem 1.7 and subsequent remarks]{lieb91} (see also \cite[Theorem 1]{lieb88} and \cite{Antonini1})   that there exist two constants $\beta\in (0,1)$ and $c_1>0$, independent of $\e$, such that $ u_\e\in C^{1,\beta}(\overline{\Omega})$ and
\begin{equation}\label{c1a:unif}
\|u_\e\|_{C^{1,\beta}(\overline{\Omega})}\leq c_1\,.
\end{equation}
Thereby, there exists a sequence $u_{\e_l}$ and a function $v\in C^1(\overline{\Omega})$ such that $u_{\e_l}\xrightarrow{l\to\infty} v$ in $C^1(\overline{\Omega})$. On the other hand, by letting $l\to \infty$ in the weak formulation of \eqref{eq:ue}, with $\varepsilon$ replaced with $\varepsilon_\ell$, we deduce that $v$ is a weak solution to \eqref{eq:dir2}, whence $v=u$ by uniqueness. The same argument applied to any sequence $\{u_{\e_l}\}$ extracted from the family $\{u_{\e}\}$ yields \eqref{conv:C1u}.
\\
In order to prove that $u_\e\in W^{2,2}(\Omega)$, one can exploit \eqref{est:Ae} and apply \cite[Theorem 8.2] {ben:fre} (see also \cite[pp. 270-277]{ladyz} or \cite[Chapter 8.4]{giusti}). Equation  \eqref{reg:ue} is thus established.

\smallskip
\par \noindent \textit{Step 2.}
Let $f$, $\Omega$ and $u_\e$ be as in Step 1. We claim that there exists a positive constant $c=c(n,i_a,s_a,\l,\L)$ such that
\begin{align}
\label{sempre}
c\int_{\{H(\nabla u_\e)=t\}}b_\e(t)\,|\nabla H(\nabla u_\e)|\, d\H^{n-1}\leq &\int_{\{H(\nabla u_\e)
    >t\}}\frac{f^2}{a_\e\big( H(\nabla u_\e)\big)}\,dx+\,t\,\int_{\{H(\nabla u_\e)=t\}}|f|\,d\H^{n-1}
    \\
    &-\,\int_{\partial \Omega\cap \{H(\nabla u_\e)>t\}}b_\e\big(H(\nabla u_\e)\big)\,H(\nu)\,H(\nabla u_\e)\,\mathrm{tr}\,\B^H\,d\H^{n-1}\nonumber 
\end{align}   
for a.e. $t>0$. 
\\
To prove the inequality \eqref{sempre}, we fix $\e \in (0,1)$ and  construct a sequence $\{u_{\e,k}\}\subset C^\infty(\overline{\Omega})$, denoted by $\{u_k\}$ for simplicity, such that $u_k= 0$ on $\partial \Omega$, and
\begin{equation}\label{conv:uk}
    u_k\xrightarrow{k\to\infty} u_\e\quad \text{in $W^{2,2}(\Omega)\cap C^1(\overline{\Omega})$.}
\end{equation}
Define the regularized vector field
    \begin{equation}\label{def:Aed}
        \A_{\e,\d}(\xi)=\A_\e\ast\rho_\d(\xi)\quad \text{for $\xi\in \R^n$,}
    \end{equation}
where $\rho_\d$ denotes  a standard, radially symmetric convolution kernel. Then  $\A_{\e,\d}\in C^\infty(\R^n)$, $\lim_{\d\to 0}\A_{\e,\d}=\A_\e$ locally uniformly in $\R^n$, and, thanks to \eqref{est:Ae} and the properties of convolutions, 
    \begin{equation}\label{est:Aed}
   \e\,\l\,\min\{1,i_{b}\}\,\mathrm{Id} \leq \nabla_\xi \A_{\e,\d}(\xi)
\leq\e^{-1}\,\L\,\max\{1,s_{b}\}\,\mathrm{Id} \quad \text{for $\xi\in \R^n$.}
\end{equation}
Next, consider the solution $w_{\e,\d}\in W^{1,2}(\Omega)$   to the problem:
\begin{equation}\label{eq:ged}
    \begin{cases}
                -\mathrm{div}\big( \A_{\e,\d}(\nabla w_{\e,\d})\big)=f\quad& \text{in $\Omega$}
        \\
         w_{\e,\d}=0\quad& \text{on $\partial\Omega$.}
    \end{cases}
\end{equation}
Thanks to \eqref{est:Aed}, classical results tell us that
 $w_{\e,\d}\in W^{2,2}(\Om)$, and
\begin{equation}\label{H2:d}
    \|w_{\e,\d}\|_{W^{2,2}(\Om)}\leq c_0
\end{equation}
for some constant $c_0=c_0(n,i_b,s_b,\l,\L,\e,\Om,\|f\|_{L^\infty(\Om)})$, see e.g. \cite[pp. 270-277]{ladyz}, or \cite[Theorem 8.2]{ben:fre}, or \cite[Chapter 8.4]{giusti}.
By \cite[Corollary 6.1]{stampacchia}, there exists a constant $c_1$, depending on the same quantities as $c_0$, such that
\begin{equation}
\|w_{\e,\d}\|_{L^\infty(\Om)}\leq c_1.
\end{equation}
Coupling  this piece of information with
 \cite[Theorem 1]{lieb88} entails that
\begin{equation}\label{unif:indeltag}
\|w_{\e,\d}\|_{C^{1,\theta}(\overline{\Om})}\leq c_2\,,
\end{equation}
for some constants $c_2$ and $\theta\in (0,1)$, with the same dependence as $c_0$ and $c_1$. In particular, these constants are independent of $\delta$. From \eqref{eq:ged},\eqref{H2:d}, \eqref{unif:indeltag} and standard elliptic regularity theory, it follows that
\begin{equation}\label{zxx}
    w_{\e,\d}\in C^{\infty}(\ov{\Om})\,.
\end{equation}
\\ Also, thanks to the inequalities \eqref{H2:d} and \eqref{unif:indeltag}, there exist a function $w_\e\in W^{2,2}(\Om)\cap C^{1,\theta}(\overline{\Om})$ and a sequence $\{\d_l\}$ such that $\d_l \to 0^+$,
\begin{equation}\label{zx}
   w_{\e,\d_l}  \xrightarrow[]{l\to \infty} w _\e \quad\text{in $C^{1,\theta'}(\overline{\Om})$} \quad  \text{and} \quad  
w_{\e,\d_l}   \xrightarrow[]{l\to\infty} w _\e
\quad \text{weakly in $W^{2,2}(\Om)$}
\end{equation}
for every $\theta ' \in (0,\theta)$. Passing to the limit as  $l\to \infty$ in the weak formulation of the  problem \eqref{eq:ged} shows that $w_\e$ is a solution to the problem \eqref{eq:ue}, whence $w_\e=u_\e$ by the uniqueness of such a solution. 
Owing to the Banach-Saks theorem, the weak convergence \eqref{zx} in the Hilbert space $W^{2,2}(\Omega)$ ensures that there exists a 
subsequence $\{\d_{l_j}\}$ such that, on setting
$$ u_k=u_{\e,k}= \frac{1}{k}\sum_{j=1}^k w_{\e,\d_{l_j}},$$
one has that
\begin{equation}\label{BS1}
     u_{k}  \xrightarrow[]{k\to \infty}  u_\e\quad\text{in $W^{2,2}(\Om)\cap C^{1,\theta'}(\overline{\Omega})$.}
\end{equation}
Equation  \eqref{conv:uk} hence follows.
\\
Now, set
\begin{align}\label{eq:uk}
 f_k=-\mathrm{div}\big( \A_\e(\nabla u_k)\big), 
\end{align}
and recall that $u_k=0$ on $\partial \Omega$. Since $|\{H(\nabla u_k)=t\}|=0$ for a.e. $t>0$ by Sard's theorem, we have $\partial^e\big(\{H(\nabla u_k)>t\}= \partial^e\big(\{H(\nabla u_k)\geq t\}$ for a.e. $t>0$. 
The following chain holds, up to sets of $\mathcal H^{n-1}$ measure zero, for a.e. $t>0$:
\begin{align}\label{bdr:0bis}
   \partial^*\{H(\nabla u_k)>t\}&= 
   \partial^e\{H(\nabla u_k)>t\}
   \\ \nonumber &=\Big(\Omega \cap \partial^e \big\{ H(\nabla u_k)>t\big\}\Big)\cup\Big(\partial \Omega\cap\partial^e \big\{ H(\nabla u_k)\geq t\big\}\Big)
    \\
    &= \big\{x\in \Omega: H(\nabla u_k(x))=t\big\}\cup \big\{x\in \partial \Omega:\,H(\nabla u_k(x))\geq t\big\}
    \nonumber
    \\
    &=  \big\{x\in \Omega: H(\nabla u_k(x))=t\big\}\Big)\cup \big\{x\in \partial \Omega:\,H(\nabla u_k(x))> t\big\}
    \nonumber
       \\ \nonumber
    &=  \big\{ H(\nabla u_k)=t\big\}\cup \Big(\partial \Omega \cap \big\{H(\nabla u_k)> t\big\}\Big)
\end{align}
 Notice that here we have  made use of   the fact that, by \cite[Theorem 16.2]{maggi}, $\H^{n-1}(\partial^e\Omega\setminus\partial^*\Omega)=0$,
 and a property of any Sobolev function  $v$  which ensures that,  up to sets of $\mathcal H^{n-1}$ measure zero,
 \begin{align}
     \label{lug2}
     \Omega \cap \partial^e \big\{ v>t\big\}= \{x\in \Omega: v(x)=t\big\} \quad \text{for a.e. $t>0$,}
 \end{align}
 see \cite{br:ziemer}.
 We have also exploited 
  a property of traces \cite[Lemma 9.5.1.2]{maz} and a version of \eqref{t:lebzero} for functions defined on $\partial \Omega$, which, in turn, is a consequence of a version of the coarea formula on $\partial \Omega$ (\cite[Chapter 2, Theorem 7.3]{simon}) to deduce that, if $v\in C^1(\overline \Omega)$, then
\begin{align}
    \label{lug3}
    \partial \Omega\cap\partial^e \big\{ v> t\big\} & = \partial \Omega\cap\partial^e \big\{ v\geq t\big\}= \big\{x\in \partial \Omega:\,v(x)\geq t\big\}\\ \nonumber &=  \big\{x\in \partial \Omega:\,v(x)> t\big\} = \partial \Omega \cap \{v>t\}\quad \text{for a.e. $t>0$,}
\end{align}
up to sets of $\mathcal H^{n-1}$ measure zero.
  \par
\noindent In the remaining part of this proof, we shall we shall consider, without further mentioning, only values of $t>0$ for which the properties \eqref{bdr:0bis}--\eqref{lug3} hold.
\\
We can thus apply Lemma \ref{quick:form-new}, with $\Omega$ replaced with $\{H(\nabla u_k)>t\}$, and with $V=\A_\e(\nabla u_k)$, $W=\tfrac{1}{2}\nabla_\xi H^2(\nabla u_k)$. Hence, thanks to \eqref{bdr:0bis}, we obtain:
\begin{align}
\label{tr1}
        -\int_{\{H(\nabla u_k)>t\}}& f_k \,\mathrm{div}\big(\tfrac{1}{2}\nabla_\xi H^2(\nabla u_k)\big)\,dx=\int_{\{H(\nabla u_k)>t\}}\mathrm{div}\big(\A_\e(\nabla u_k) \big)\mathrm{div}\big(\tfrac{1}{2}\nabla_\xi H^2(\nabla u_k)\big)\,dx
        \\
        =&\int_{\{H(\nabla u_k)>t\}}\mathrm{tr}\Big(\big[\nabla\A_\e(\nabla u_k)\big]\big[ \nabla\big( \tfrac{1}{2}\nabla_\xi H^2(\nabla u_k)\big)\big]\Big)\,dx\nonumber
        \\
    &+\int_{\partial^* \{H(\nabla u_k)>t\}}\Big(\big(\mathrm{div} [\tfrac{1}{2}\nabla_\xi H^2(\nabla u_k)]\big)\,\A_\e(\nabla u_k)\cdot\nu- \nabla [\tfrac{1}{2}\nabla_\xi H^2(\nabla u_k)]\, [\A_\e(\nabla u_k)] \cdot \nu \Big)\,d\mathcal{H}^{n-1}\,\nonumber
\\
        =&\int_{\{H(\nabla u_k)>t\}}\mathrm{tr}\Big(\big[\nabla\A_\e(\nabla u_k)\big]\big[ \nabla\big( \tfrac{1}{2}\nabla_\xi H^2(\nabla u_k)\big)\big]\Big)\,dx\nonumber
        \\
    &+\int_{\{H(\nabla u_k)=t\}}\Big(\big(\mathrm{div} [\tfrac{1}{2}\nabla_\xi H^2(\nabla u_k)]\big)\,\A_\e(\nabla u_k)\cdot\nu- \nabla [\tfrac{1}{2}\nabla_\xi H^2(\nabla u_k)]\, [\A_\e(\nabla u_k)] \cdot \nu \Big)\,d\mathcal{H}^{n-1}\,\nonumber
     \\
&+\int_{\partial \Omega \cap \{H(\nabla u_k)>t\}}\Big(\big(\mathrm{div} [\tfrac{1}{2}\nabla_\xi H^2(\nabla u_k)]\big)\,\A_\e(\nabla u_k)\cdot\nu- \nabla [\tfrac{1}{2}\nabla_\xi H^2(\nabla u_k)]\, [\A_\e(\nabla u_k)] \cdot \nu \Big)\,d\mathcal{H}^{n-1}\,.\nonumber
\end{align}
We analyze each term  of the rightmost side of the above chain separately.  To begin with, recall that,  given any  symmetric, positive semi-definite matrices $X$ and $Y$, 
\begin{equation}\label{tr:in}
    \mathrm{tr}(XY)\geq \l_{min}\,\mathrm{tr} Y,
\end{equation}
 where $\l_{min}>0$ denotes the smallest eigenvalue of $X$. Hence, via the chain rule for vector valued functions (see \cite{mm}), and Equations \eqref{ell:H},   \eqref{pr:ae}, and \eqref{tr:in}, we deduce that
\begin{align}
\label{tr2}
        \mathrm{tr}\Big( \nabla \big(\tfrac{1}{2}\nabla_\xi H^2(\nabla u_k)\big)\,\nabla \big(\mathcal{A}_\e(\nabla u_k)\big)\Big)&=\mathrm{tr}\Big( \tfrac{1}{2}\,\nabla_\xi^2 H^2(\nabla u_k)\,\nabla^2 u_k\,\nabla_\xi \mathcal{A}_\e(\nabla u_k)\,\nabla^2 u_k\Big)
        \\
        &\geq \lambda\,\mathrm{tr}\Big(\nabla^2 u_k\,\nabla_\xi \mathcal{A}_\e(\nabla u_k)\,\nabla^2 u_k\Big)\nonumber
        \\
        &=\lambda\,\mathrm{tr}\Big(\nabla_\xi \mathcal{A}(\nabla u_k)\,\nabla^2 u_k\,\nabla^2 u_k\Big)\nonumber
        \\
        &\geq \lambda^2\,a_\e\big(H(\nabla u_k)\big)\,\mathrm{tr}\big(\nabla^2 u_k\,\nabla^2 u_k\big)\nonumber
        \\
        &= \lambda^2\,a_\e\big(H(\nabla u_k)\big)\,|\nabla^2 u_k|^2,\nonumber
\end{align}
where $|\nabla^2 u_k|^2=\sum_{i,j=1}^n |\partial_{i,j} u_k|^2$. Notice that the matrix $\nabla^2 u_k\,\nabla_\xi \mathcal{A}_\e(\nabla u_k)\,\nabla^2 u_k$ is actually symmetric and positive definite and $\nabla^2 u_k \,\nabla^2 u_k$ is symmetric and positive semi-definite, since 
$\mathcal{A}_\e(\nabla u_k)$ is symmetric and positive definite, and
$\nabla^2 u_k$ is symmetric. 
\\
By the chain rule, \eqref{ell:H} and Young's inequality, 
\begin{align}
    \label{trr22}
    & \int_{\{H(\nabla u_k)>t\}} f_k \,\mathrm{div}\big(\tfrac{1}{2}\nabla_\xi H^2(\nabla u_k)\big)\,dx  = \int_{\{H(\nabla u_k)>t\}} f_k\,\mathrm{tr}\Big(\tfrac{1}{2}\nabla^2_\xi H^2(\nabla u_k) \,\nabla^2 u_k\Big)\,dx
    \\
    &\hspace{0.3cm} \leq c(n)\,\Lambda\int_{\{H(\nabla u_k)>t\}} |f_k|\,|\nabla^2 u_k|\,dx\nonumber
    \\
    &\hspace{0.3cm}\leq\frac{\l^2}{2}\, \int_{\{H(\nabla u_k)>t\}}\,a_\e\big(H(\nabla u_k)\big)\,|\nabla^2 u_k|^2 dx+c'(n)\frac{\L^2}{\l^2}\int_{\{H(\nabla u_k)>t\}} \frac{f_k^2}{a_\e\big(H(\nabla u_k)\big)}\,dx\,.\nonumber
\end{align}
Combining \eqref{tr1}, \eqref{tr2}, and \eqref{trr22}  yields:
\begin{align}\label{temp:1}
    & c\,\int_{\{H(\nabla u_k)>t\}} \frac{f_k^2}{a_\e\big( H(\nabla u_k)\big)}\,dx  \geq \frac{\l^2}{2}\,\int_{\{H(\nabla u_k)>t\}}\,a_\e\big(H(\nabla u_k)\big)\,|\nabla^2 u_k|^2 dx
\\
    &+\int_{\{H(\nabla u_k)=t\}}\Big(\big(\mathrm{div} [\tfrac{1}{2}\nabla_\xi H^2(\nabla u_k)]\big)\,\A_\e(\nabla u_k)\cdot\nu- \nabla [\tfrac{1}{2}\nabla_\xi H^2(\nabla u_k)]\, [\A_\e(\nabla u_k)] \cdot \nu \Big)\,d\mathcal{H}^{n-1}\,\nonumber
     \\
&+\int_{\partial \Omega \cap \{H(\nabla u_k)>t\}}\Big(\big(\mathrm{div} [\tfrac{1}{2}\nabla_\xi H^2(\nabla u_k)]\big)\,\A_\e(\nabla u_k)\cdot\nu- \nabla [\tfrac{1}{2}\nabla_\xi H^2(\nabla u_k)]\, [\A_\e(\nabla u_k)] \cdot \nu \Big)\,d\mathcal{H}^{n-1}\,.\nonumber
\end{align}
Consider the integral over $\{H(\nabla u_k)=t\}$ on  the right-hand side of \eqref{temp:1}. For a.e. $t>0$, the outward unit normal vector $\nu$ to $\{H(\nabla u_k)=t\}$ is given by
\begin{equation}\label{nu:gamma0}
    \nu=-\frac{\nabla H(\nabla u_k)}{|\nabla H(\nabla u_k)|} \quad\text{ $\mathcal H^{n-1}$-a.e. on $\{H(\nabla u_k)=t\}$}.
\end{equation}
Also, by \eqref{eq:uk} and the chain rule, 
\begin{align}\label{tr:f}
    -f_k&=\mathrm{div}\big(\A_\e(\nabla u_k)\big)  =\mathrm{div}\big(a_\e(H(\nabla u_k))\,\tfrac{1}{2}\nabla_\xi H^2(\nabla u_k)\big)
    \\
    &=a_\e(H(\nabla u_k))\,\mathrm{div}\big(\tfrac{1}{2}\nabla_\xi H^2(\nabla u_k)\big)+a'_\e(H(\nabla u_k))\,\nabla H(\nabla u_k)\cdot \tfrac{1}{2}\nabla_\xi H^2(\nabla u_k)\,.\nonumber
\end{align}
From Equations \eqref{nu:gamma0} and \eqref{tr:f} we deduce that
\begin{align}
\label{ttrr}
    &\mathrm{div} [\tfrac{1}{2}\nabla_\xi H^2(\nabla u_k)]\big)\,\A_\e(\nabla u_k)\cdot\nu= a_\e(t)\,\mathrm{div} [\tfrac{1}{2}\nabla_\xi H^2(\nabla u_k)]\,\tfrac{1}{2}\nabla_\xi H^2(\nabla u_k)\cdot \nu 
    \\
    &= -f_k\,\tfrac{1}{2}\nabla_\xi H^2(\nabla u_k)\cdot \nu-a'_\e(t)\,\big[\nabla H(\nabla u_k)\cdot \tfrac{1}{2}\nabla_\xi H^2(\nabla u_k)\big]\,\big[\tfrac{1}{2}\nabla_\xi H^2(\nabla u_k)\cdot \nu\big]\nonumber
    \\
    &=-f_k\,\tfrac{1}{2}\nabla_\xi H^2(\nabla u_k)\cdot \nu+a'_\e(t)\,|\nabla H(\nabla u_k)|\big(\tfrac{1}{2}\nabla_\xi H^2(\nabla u_k)\cdot \nu \big)^2\nonumber
    \\
    &=-f_k\,\tfrac{1}{2}\nabla_\xi H^2(\nabla u_k)\cdot \nu+a'_\e(t)\,t^2\,|\nabla H(\nabla u_k)|\,\big(\nabla_\xi H(\nabla u_k)\cdot \nu \big)^2\nonumber
\end{align}
$\H^{n-1}$-a.e. on $\{H(\nabla u_k)=t\}$.
Thereby, by the chain rule and   \eqref{nu:gamma0}, the following chain holds:
\begin{align}\label{contt:1}
        \nabla [\tfrac{1}{2}\nabla_\xi H^2(\nabla u_k)]\, [\tfrac{1}{2}\nabla_\xi H^2(\nabla u_k)] \cdot \nu &=\tfrac{1}{2}\partial_{\xi_i\,\xi_l}H^2(\nabla u_k)\,\partial_{x_l x_j} u_k\,\tfrac{1}{2}\partial_{\xi_j} H^2(\nabla u_k)\,\nu_i
        \\
        & =\tfrac{1}{2}\partial_{\xi_i\,\xi_l}H^2(\nabla u_k)\,\tfrac{1}{2}\partial_{x_l} H^2(\nabla u_k)\,\nu_i\nonumber
        \\
        &=\tfrac{1}{2}\partial_{\xi_i\,\xi_l}H^2(\nabla u_k)\,H(\nabla u_k)\,\partial_{x_l} H(\nabla u_k)\,\nu_i\nonumber
        \\ \nonumber
        &=-t\,|\nabla H(\nabla u_k)|\,\tfrac{1}{2}\nabla_\xi^2 H^2(\nabla u_k)\nu\cdot \nu
    \end{align}
    $\H^{n-1}$-a.e. on $\{H(\nabla u_k)=t\}$.
From Equations \eqref{nu:gamma0}, \eqref{ttrr}, and \eqref{contt:1} one infers that 
\begin{align}\label{contt:2}
         \mathrm{div} &\,[\tfrac{1}{2}\nabla_\xi H^2(\nabla u_k)]\big)\,\A_\e(\nabla u_k)\cdot\nu- \nabla [\tfrac{1}{2}\nabla_\xi H^2(\nabla u_k)]\, [\A_\e(\nabla u_k)] \cdot \nu 
        \\
        &=-f_k\,\tfrac{1}{2}\nabla_\xi H^2(\nabla u_k)\cdot \nu+a'_\e(t)\,t^2\,|\nabla H(\nabla u_k)|\,\big(\nabla_\xi H(\nabla u_k)\cdot \nu \big)^2\nonumber
        \\
        &\hspace{2.9cm} - a_\e\big(H(\nabla u_k)\big)\,\nabla [\tfrac{1}{2}\nabla_\xi H^2(\nabla u_k)]\, [\tfrac{1}{2}\nabla_\xi H^2(\nabla u_k)] \cdot \nu\nonumber
        \\
        &=-f_k\,t\,\nabla_\xi H(\nabla u_k)\cdot \nu+a'_\e(t)\,t^2\,|\nabla H(\nabla u_k)|\,\big(\nabla_\xi H(\nabla u_k)\cdot \nu \big)^2\nonumber
        \\ \nonumber
        &\hspace{3.7cm} + a_\e(t)\,t\,|\nabla H(\nabla u_k)|\,\tfrac{1}{2}\nabla_\xi^2 H^2(\nabla u_k)\nu\cdot \nu 
    \end{align}
    $\H^{n-1}$-a.e. on $\{H(\nabla u_k)=t\}$. 
Observe that
\begin{equation*}
    \tfrac{1}{2}\nabla_\xi^2 H^2(\nabla u_k)\nu\cdot \nu=t\,\nabla_\xi^2 H(\nabla u_k )\,\nu\cdot \nu+(\nabla_\xi H(\nabla u_k)\cdot \nu)^2
\end{equation*}
$\H^{n-1}$-a.e. on $\{H(\nabla u_k)=t\}$.
The latter identity, in combination with \eqref{b_e}, \eqref{c:a} and \eqref{appr:2}, results in:
\begin{align}\label{giochi:norma}
    & a'_\e(t)\,t^2\,\big(\nabla_\xi H(\nabla u_k)\cdot \nu \big)^2+ a_\e(t)\,t\,\tfrac{1}{2}\nabla_\xi^2 H^2(\nabla u_k)\nu\cdot\nu
    \\
    &=a_\e(t)\,t\,\Big\{t\,\nabla_\xi^2 H(\nabla u_k)\nu\cdot\nu+\Big[  1+\frac{a_\e'(t)\,t}{a_\e(t)}\Big](\nabla_\xi H(\nabla u_k)\cdot \nu)^2\Big\}\nonumber
    \\
    &\geq b_\e(t)\,\Big\{t\,\nabla_\xi^2 H(\nabla u_k)\nu\cdot\nu+(  1+i_{a_\e})(\nabla_\xi H(\nabla u_k)\cdot \nu)^2\Big\}\nonumber
    \\
    &\geq \min\{1,1+i_{a_\e}\}\,b_\e(t)\,\tfrac{1}{2}\nabla^2_\xi H^2(\nabla u_k)\,\nu\cdot\nu\qquad\text{$\H^{n-1}$-a.e. on $\{H(\nabla u_k)=t\}$.}\nonumber
    \end{align}
Coupling  \eqref{contt:2} with \eqref{giochi:norma} yields:
\begin{align}\label{cont:3}
        \mathrm{div} [\tfrac{1}{2}&\,\nabla_\xi H^2(\nabla u_k)]   \,\A_\e(\nabla u_k)\cdot\nu- \nabla [\tfrac{1}{2}\nabla_\xi H^2(\nabla u_k)]\, [\A_\e(\nabla u_k)] \cdot \nu
        \\
        &\geq -t\,f_k\,\nabla_\xi H(\nabla u_k)\cdot \nu+\min\{1,1+i_{a_\e}\}\,b_\e(t)\,|\nabla H(\nabla u_k)|\,\tfrac{1}{2}\nabla^2_\xi H^2(\nabla u_k)\,\nu\cdot\nu,\nonumber
    \end{align}
$\H^{n-1}$-a.e. on $\{H(\nabla u_k)=t\}$.
\\ On the other hand, 
from Equations \eqref{gen21},
 \eqref{ag:24}, and  \eqref{lug3}, we have that 
\begin{align}\label{cont:4}
 \mathrm{div}& [\tfrac{1}{2}\nabla_\xi H^2(\nabla u_k)]\,\A_\e(\nabla u_k)\cdot\nu-  \nabla [\tfrac{1}{2}\nabla_\xi H^2(\nabla u_k)]\, [\A_\e(\nabla u_k)] \cdot \nu
\\
  & = \mathrm{div}_T [\tfrac{1}{2}\nabla_\xi H^2(\nabla u_k)]\,\A_\e(\nabla u_k)\cdot\nu  -  \nabla_T [\tfrac{1}{2}\nabla_\xi H^2(\nabla u_k)]\, [\A_\e(\nabla u_k)]_T \cdot \nu\nonumber
    \\
    &=a_\e\big( H(\nabla u_k)\big)\,H(\nu)\,H^2(\nabla u_k)\,\mathrm{tr}\,\B^H\nonumber
=b_\e\big( H(\nabla u_k)\big)\,H(\nu)\,H(\nabla u_k)\,\mathrm{tr}\,\B^H\nonumber
    \end{align}
   $\H^{n-1}$-a.e. on $\{H(\nabla u_k)=t\}$. 
Combining \eqref{cont:3}, \eqref{cont:4}, \eqref{temp:1} and \eqref{bdr:0bis} yields
\begin{align}\label{lug1}
        c\,\int_{\{H(\nabla u_k)>t\}} \frac{f_k^2}{a_\e\big( H(\nabla u_k)\big)}\,dx & \geq \frac{\l^2}{2}\,\int_{\{H(\nabla u_k)>t\}} a_\e\big( H(\nabla u_k)\big)\,|\nabla^2 u_k|^2\,dx+
        \\
       &-t\,\int_{\{H(\nabla u_k)=t\}} f_k\,\nabla_\xi H(\nabla u_k)\cdot \nu\,d\H^{n-1}\nonumber
        \\
        &+\min\{1,1+i_{a_\e}\}\,b_\e(t)\,\int_{\{H(\nabla u_k)=t\}} \frac{1}{2}\nabla^2_\xi H^2(\nabla u_k)\,\nu\cdot\nu\,|\nabla H(\nabla u_k)|\, d\H^{n-1}\nonumber
        \\
        &+\int_{ \partial \Omega\cap \{H(\nabla u_k)>t\} } b_\e\big(H(\nabla u_k)\big)\,H(\nu)\,H(\nabla u_k)\,\mathrm{tr}\,\B^H\,d\H^{n-1}\nonumber
\end{align}
for a.e. $t>0$. 
By Equation \eqref{ell:H} and the inequality $|\nabla H(\nabla u_k)\cdot \nu|\leq \sqrt{L}$, which holds thanks  to \eqref{bound:nabH1},  the inequality \eqref{lug1} implies that
\begin{align}\label{temp:uk}
        c\,\int_{\{H(\nabla u_k)>t\}} \frac{f_k^2}{a_\e\big( H(\nabla u_k)\big)}\,dx\geq & \,\l\,\min\{1,1+i_{a_\e}\}\,b_\e(t)\,\int_{\{H(\nabla u_k)=t\}} |\nabla H(\nabla u_k)|\, d\H^{n-1}
        \\
        &-\sqrt{\L}\,t\,\int_{\{H(\nabla u_k)=t\}} |f_k|\,d\H^{n-1}\nonumber
        \\
    &+\int_{
    {\partial \Omega\cap \{H(\nabla u_k)>t\}}
    %
    } b_\e\big(H(\nabla u_k)\big)\,H(\nu)\,H(\nabla u_k)\,\mathrm{tr}\,\B^H\,d\H^{n-1}\nonumber
    \end{align}
for some constant  $c=c(n,i_a,s_a,\l,\L)>0$ and  for a.e. $t>0$.
\\ Our next task consists in passing to the limit as $k\to \infty$ in \eqref{temp:uk}. Since $a_\varepsilon\in C^1([0,\infty)) $, Equations \eqref{reg:H}, \eqref{def:Ae},  \eqref{reg:ue}, and the chain rule \cite{mm} ensure that 
$$\A_\e(\nabla u_\e)\in W^{1,2}(\Omega)\cap C^0(\overline{\Omega})\quad\text{and}\quad H(\nabla u_\e)\in W^{1,2}(\Omega)\cap C^0(\overline{\Omega}).
$$
As a consequence of Equation \eqref{coarea:1}, applied with $u$ replaced with $H(\nabla u_\e)$,
\begin{equation}\label{coarea1:ue}
    -\frac{d}{dt}\int_{\{H(\nabla u_\e)>t\}}g\,|\nabla H(\nabla u_\e)|\,dx=\int_{\{H(\nabla u_\e)=t\}}g\,d\H^{n-1}\quad\text{for a.e. $t>0$.}
\end{equation}
Also, owing to \eqref{t:hn1zero} and\eqref{t:lebzero},  
\begin{equation}\label{hn1:zeroue}
    \H^{n-1}\big( \{H(\nabla u_\e)=t\}\cap \{|\nabla H(\nabla u_\e)|=0\}\big)=0\quad\text{for a.e. $t>0$}\,,
\end{equation}
and
\begin{equation}\label{leb:zero_ue}
    |\{H(\nabla u_\e)=t\}|=0\quad\text{for a.e. $t>0$.}
\end{equation}
Given $h>0$, 
we integrate the inequality \eqref{temp:uk} over $(t,t+h)$  and use the coarea formula \eqref{coarea}    to obtain:
\begin{align}\label{equa:int}
         \int_{\{t<H(\nabla u_k)<t+h\}}&\,b_\e\big(H(\nabla u_k)\big)\, |\nabla H(\nabla u_k)|^2\, dx
         \\
         \leq &\,c\,\int_t^{t+h}       \int_{\{H(\nabla u_k)>s\}} \frac{f_k^2}{a_\e\big( H(\nabla u_k)\big)}\,dx\,ds\nonumber
        \\
        &+c\,\int_{\{t<H(\nabla u_k)<t+h\}} H(\nabla u_k) |f_k|\,|\nabla H(\nabla u_k)|\,dx\nonumber
        \\
    &-c\,\int_t^{t+h}\int_{\partial \Omega\cap \{H(\nabla u_k)>s\}} b_\e\big(H(\nabla u_k)\big)\,H(\nu)\,H(\nabla u_k)\,\mathrm{tr}\,\B^H\,d\H^{n-1}\,ds\,,\nonumber
\end{align}
for some positive constant $c=c(n,i_a,s_a,\l,\L)$.
\\
Thanks to \eqref{reg:H}, \ref{appr:3} and \eqref{conv:uk}, one can deduce from \cite[Proposition 2.6, Example 3.2]{musina} that
\begin{equation}\label{convergenzaa}
    \A_\e(\nabla u_k)\xrightarrow{k\to\infty} \A_\e(\nabla u_\e)\quad\text{and} \quad H(\nabla u_k)\xrightarrow{k\to\infty} H(\nabla u_\e)\quad\text{in $W^{1,2}(\Omega)\cap C^0(\overline{\Omega})$.}
\end{equation}
In particular, $f_k\to f$ in $L^2(\Omega)$.
\\ Now, 
Equations \eqref{conv:uk}, \eqref{convergenzaa}
and the properties of $a_\e$ collected in Section \ref{sec:tutto}
enable us to pass to the limit as $k\to\infty$ in \eqref{equa:int} and deduce, via a standard variant of Lebesgue dominated convergence theorem as in \cite[Theorem 3]{Antonini}, that
\begin{align*}
         \int_{\{t<H(\nabla u_\e)<t+h\}}\,b_\e\big(H(\nabla u_\e)\big)\, |\nabla H(\nabla u_\e)&|^2\, dx 
         \leq \,c\,\int_t^{t+h}       \int_{\{H(\nabla u_\e)>s\}} \frac{f^2}{a_\e\big( H(\nabla u_\e)\big)}\,dx\,ds\nonumber
        \\
        &+c\,\int_{\{t<H(\nabla u_\e)<t+h\}} H(\nabla u_\e) |f|\,|\nabla H(\nabla u_\e)|\,dx\nonumber
        \\
    &-c\,\int_t^{t+h}\int_{\partial \Omega\cap \{H(\nabla u_\e)>s\}} b_\e\big(H(\nabla u_\e)\big)\,H(\nu)\,H(\nabla u_\e)\,\mathrm{tr}\,\B^H\,d\H^{n-1}\,ds\,,\nonumber
    \end{align*}
for a.e. $t>0$. Notice that  we have also used \eqref{leb:zero_ue} to handle the integral over $\partial \Omega$. By the coarea formula,
\begin{align*}
     \int_{\{t<H(\nabla u_\e)<t+h\}} &\,b_\e\big(H(\nabla u_\e)\big)\, |\nabla H(\nabla u_\e)|^2\, dx
     =\int_t^{t+h}\int_{\{H(\nabla u_\e)=s \}}\,b_\e\big(H(\nabla u_\e)\big)\, |\nabla H(\nabla u_\e)|\,d\H^{n-1}\,ds\,,
\end{align*}
and
\begin{align*}
    \int_{\{t<H(\nabla u_\e)<t+h\}}&\, H(\nabla u_\e) |f|\,|\nabla H(\nabla u_\e)|\,dx =\int_t^{t+h}\int_{\{ H(\nabla u_\e)=s\}}H(\nabla u_\e) |f|\,d\H^{n-1}\,ds
\end{align*}
for $t, h>0$. 
From  the last three equations one has that
\begin{align}\label{ciccia}
        \int_t^{t+h}\int_{\{H(\nabla u_\e)=s \}}&\,b_\e\big(H(\nabla u_\e)\big)\, |\nabla H(\nabla u_\e)|\,d\H^{n-1}\,ds
         \\
         \leq &\,c\,\int_t^{t+h}       \int_{\{H(\nabla u_\e)>s\}} \frac{f^2}{a_\e\big( H(\nabla u_\e)\big)}\,dx\,ds\nonumber
        \\
        &+c\,\int_t^{t+h}\int_{\{ H(\nabla u_\e)=s\} }H(\nabla u_\e) |f|\,d\H^{n-1}\,ds\nonumber
        \\
    &-c\,\int_t^{t+h}\int_{\partial \Omega\cap \{H(\nabla u_\e)>s\}} b_\e\big(H(\nabla u_\e)\big)\,H(\nu)\,H(\nabla u_\e)\,\mathrm{tr}\,\B^H\,d\H^{n-1}\,ds\,,\nonumber
\end{align}
Dividing through by $h$ in the inequality \eqref{ciccia} and the coarea formula \eqref{coarea1:ue} enable one to deduce \eqref{sempre}.

\smallskip
\noindent\textit{Step 3.} This step is devoted to estimating the terms appearing in the inequality
 \eqref{sempre}. Denote by $\mu : (0, \infty) \to [0, \infty)$, the  distribution function of $H(\nabla u_\e)$; namely, 
\begin{equation*}
    \mu(t)=|\{x\in \Omega:\,H(\nabla u_\e(x))>t\}| \quad \text{for $t>0$.}
\end{equation*}
 Equation \eqref{rel:isoper}  implies that
\begin{equation}\label{rel:isopH}
    \mu(t)^{1/n'}\leq c\,\H^{n-1}\big(\{H(\nabla u_\e)=t\} \big)\quad\text{for a.e. $t\geq H(\nabla u_\e)^*(|\Omega|/2)$,}
\end{equation}
for some constant $c=c(n,d_\Omega,\mathcal{L}_\Omega)$.
Since the function $b_\e$ is non-decreasing, and,  by \eqref{bounds:H},  $H(\nu)\leq \sqrt{\L}$,  from \eqref{equiv:BBH} and the Hardy-Littlewood inequality \eqref{HL} we infer that
\begin{align}\label{start:deom}
    \int_{\partial \Omega\cap \{H(\nabla u_\e)>t\}} &\,b_\e\big(H(\nabla u_\e)\big)\,H(\nu)\,H(\nabla u_\e)\,\mathrm{tr}\,\B^H\,d\H^{n-1}
    \\
   & \leq c\,b_\e\big(\|H(\nabla u_\e)\|_{L^\infty(\Omega)} \big)\,\|H(\nabla u_\e)\|_{L^\infty(\Omega)}\,\int_{\partial \Omega\cap \{H(\nabla u_\e)>t\}} |\B|\,d\H^{n-1}\nonumber
   \\
   &\leq c\,b_\e\big(\|H(\nabla u_\e)\|_{L^\infty(\Omega)} \big)\,\|H(\nabla u_\e)\|_{L^\infty(\Omega)}\, \int_0^{\H^{n-1}\big(\partial \Omega\cap \{H(\nabla u_\e)>t\}\big)}|\B|^*(r) dr\,\nonumber
\end{align}
for a.e. $t>0$ and for some constant $c=c(n,\l,\L,i_a,s_a)$. Here, $|\B|^*$ denotes the decreasing rearrangement of $|\B|$ with respect to $\mathcal H^{n-1}$ on $\partial \Omega$.
By \eqref{rel:isoper1}, \eqref{lug2} and \eqref{lug3},
\begin{equation}\label{nonso}
    \H^{n-1}\big(\partial \Omega\cap \{H(\nabla u_\e)>t\}\big)\leq c\,\H^{n-1}\big( \{H(\nabla u_\e)=t\}\big)\,,
\end{equation}
for a.e. $t>0$ such that $t\geq H(\nabla u_\e)^*(|\Omega|/2)$, where $c=c(n,d_\Omega, \mathcal{L}_\Omega)$.
\\ By \eqref{rel:isopH}, \eqref{nonso} and the monotonicity of the function $|\B|^{**}$, one has that
\begin{align*}
        \int_0^{\H^{n-1}\big(\partial \Omega\cap \{H(\nabla u_\e)>t\}\big)}|\B|^*(r) dr & \leq \int_0^{c\,\H^{n-1}(\{H(\nabla u_\e)=t\})}|\B|^*(r) dr
        \\
        & =c\,\H^{n-1}(\{H(\nabla u_\e)=t\})\,|\B|^{**}\big(c\,\H^{n-1}(\{H(\nabla u_\e)=t\}) \big)
        \\
        &\leq c\,\H^{n-1}(\{H(\nabla u_\e)=t\})\,|\B|^{**}\big(c'\,\mu(t)^{1/n'} \big)\quad \text{for a.e. $t>0$,}
    \end{align*}
and for some positive constants $c,c'$ depending on $n,d_\Omega,\mathcal{L}_\Omega$.
Coupling the latter inequality with \eqref{start:deom} implies that
\begin{align}\label{continua:deom}
    \int_{\partial \Omega\cap \{H(\nabla u_\e)>t\}} &\,b_\e\big(H(\nabla u_\e)\big)\,H(\nu)\,H(\nabla u_\e)\,\mathrm{tr}\,\B^H\,d\H^{n-1}
    \\
    &\leq c\,b_\e\big(\|H(\nabla u_\e)\|_{L^\infty(\Omega)} \big)\,\|H(\nabla u_\e)\|_{L^\infty(\Omega)}\, \H^{n-1}\big(\{H(\nabla u_\e)=t\}\big)\,|\B|^{**}\big(c'\,\mu(t)^{1/n'} \big)\nonumber
\end{align}
for a.e. $t>0$ and
for some positive constants $c,c'$ depending on $n,\l,\L,i_a,s_a, d_\Omega,\mathcal{L}_\Omega$.
\\ Owing
to the coarea formula \eqref{coarea1:ue},  
\begin{equation}\label{start:0}
    b_\e(t)\,\int_{\{H(\nabla u_\e)=t\}}|\nabla H(\nabla u_\e)|\, d\H^{n-1}=b_\e(t)\,\bigg\{-\frac{d}{dt}\int_{\{H(\nabla u_\e>t)\}} |\nabla H(\nabla u_\e)|^2\,dx\bigg\}
\end{equation}
for a.e. $t>0$. Moreover, 
by the monotonicity of the function $b_\e$ and the Hardy-Littlewood inequality \eqref{HL},  
\begin{align}\label{start:1}
    \int_{\{H(\nabla u_\e)
    >t\}}\frac{f^2}{a_\e\big( H(\nabla u_\e)\big)}\,dx &\leq \frac{\|H(\nabla u_\e)\|_{L^\infty(\Omega)}}{b_\e(t)}\,\int_{\{H(\nabla u_\e)
    >t\}} f^2\,dx
        \\
        &\leq \frac{\|H(\nabla u_\e)\|_{L^\infty(\Omega)}}{b_\e(t)}\,\int_0^{\mu(t)} (f^*)^2(r)\,dr\nonumber
\end{align}
for a.e. $t>0$.
Next, recall
\eqref{hn1:zeroue}, and use H\"older's inequality and the coarea formula \eqref{coarea1:ue} to deduce that
\begin{align}\label{start:2}
        \int_{\{H(\nabla u_\e)=t\}}|f|\,d\H^{n-1} & \leq \bigg(\int_{\{H(\nabla u_\e)=t\}}\frac{f^2}{|\nabla H(\nabla u_\e)|} d\H^{n-1}\bigg)^{1/2}\bigg( \int_{\{H(\nabla u_\e)=t\}}|\nabla H(\nabla u_\e)|d\H^{n-1}\bigg)^{1/2}
        \\
        &= \bigg(-\frac{d}{dt}\int_{\{H(\nabla u_\e)>t\}} f^2\,dx\bigg)^{1/2}\,\bigg(-\frac{d}{dt}\int_{\{H(\nabla u_\e)>t\}} |\nabla H(\nabla u_\e)|^2\,dx\bigg)^{1/2}\,,\nonumber
    \end{align}
for a.e. $t>0$.
Similarly, one has that
\begin{align}\label{start:3}
    \H^{n-1}(\{H(\nabla u_\e)=t\}) & \leq \bigg(\int_{\{H(\nabla u_\e)=t\}}\frac{1}{|\nabla H(\nabla u_\e)|} d\H^{n-1}\bigg)^{1/2}\bigg( \int_{\{H(\nabla u_\e)=t\}}|\nabla H(\nabla u_\e)|d\H^{n-1}\bigg)^{1/2}
    \\
    &=(-\mu'(t))^{1/2}\,\bigg(-\frac{d}{dt}\int_{\{H(\nabla u_\e)>t\}} |\nabla H(\nabla u_\e)|^2\,dx\bigg)^{1/2}\nonumber
    \end{align}
for a.e. $t>0$.
Making use of the  inequalities \eqref{continua:deom}-\eqref{start:3} and of the inequality \eqref{dis:distrib} with $v=H(\nabla u_\e)$
to estimate the various terms in  \eqref{sempre} implies that
\begin{align}\label{sempre:1}
     b_\e(t)\,\bigg\{-\frac{d}{dt}&\int_{\{H(\nabla u_\e)>t\}} |\nabla H(\nabla u_\e)|^2\,dx\bigg\} 
    \\ \nonumber
    &\leq c\,(-\mu'(t))^{1/2}\,\mu(t)^{-1/n'}\,\frac{\|H(\nabla u_\e)\|_{L^\infty(\Omega)}}{b_\e(t)}
    \\ \nonumber & \quad \quad \times \Big(\int_0^{\mu(t)} (f^*)^2(r)\,dr\Big)\,\Big(-\frac{d}{dt}\,\int_{\{H(\nabla u_\e)>t\}}|\nabla H(\nabla u_\e)|^2 dx \Big)^{1/2}\nonumber
    \\ \nonumber
    & \quad +c\,t\,\bigg(-\frac{d}{dt}\int_{\{H(\nabla u_\e)>t\}} f^2\,dx\bigg)^{1/2}\,\bigg(-\frac{d}{dt}\int_{\{H(\nabla u_\e)>t\}} |\nabla H(\nabla u_\e)|^2\,dx\bigg)^{1/2}\nonumber
    \\ \nonumber
    & \quad +c\,(-\mu'(t))^{1/2}\,b_\e\big(\|H(\nabla u_\e)\|_{L^\infty(\Omega)} \big)\,\|H(\nabla u_\e)\|_{L^\infty(\Omega)}\,|\B|^{**}\big(c'\,\mu(t)^{1/n'}\big)\\ \nonumber & \quad \quad \times \Big(-\frac{d}{dt}\int_{\{H(\nabla u_\e)>t\}}|\nabla H(\nabla u_\e)|^2 dx \Big)^{1/2}\nonumber
    \end{align}
for a.e. $t\geq H(\nabla u_\e)^*(|\Omega|/2)$ and for some constants  $c$ and $c'$, depending on $n,i_a,s_a,\l,\L,d_\Omega,\mathcal{L}_\Omega$.
\\
Dividing through in \eqref{sempre:1} by $-\frac{d}{dt}\int_{\{H(\nabla u_\e)>t\}}|\nabla H(\nabla u_\e)|^2 dx $, and the use of \eqref{dis:distrib} again with $v=H(\nabla u_\e)$ yield:
\begin{align}\label{sempre:2}
        b_\e(t)\leq &\,c\,(-\mu'(t))\,\mu(t)^{-2/n'}\,\frac{\|H(\nabla u_\e)\|_{L^\infty(\Omega)}}{b_\e(t)}\,\bigg(\int_0^{\mu(t)} (f^*)^2(r)\,dr\bigg)
        \\
        & +c\,t\,(-\mu'(t))^{1/2}\,\mu(t)^{-1/n'}\,\bigg(-\frac{d}{dt}\int_{\{H(\nabla u_\e)>t\}} f^2\,dx\bigg)^{1/2}\nonumber
        \\
        &+ c\,(-\mu'(t))\,\mu(t)^{-1/n'}\,b_\e\big(\|H(\nabla u_\e)\|_{L^\infty(\Omega)} \big)\,\|H(\nabla u_\e)\|_{L^\infty(\Omega)}\,|\B|^{**}\big(c'\,\mu(t)^{1/n'}\big)\,,\nonumber
\end{align}
for a.e. $t\geq H(\nabla u_\e)^*(|\Omega|/2)$. 
\\ Now, consider the so-called \textit{pseudo-rearrangement} of $f$ relative to $H(\nabla u_\e)$, defined as the function $\phi:(0,|\Omega|)\to [0,\infty)$ given by
\begin{equation}\label{pseudo0}
    \phi(s)=\bigg(\frac{d}{ds}\int_{\big\{H(\nabla u_\e)>H(\nabla u_\e)^*(s) \big\}}f^2 dx \bigg)^{1/2} \quad \text{for $s\in (0, |\Omega|)$}\,.
\end{equation}
By the chain rule,
\begin{equation}\label{psedo1}
    \bigg(-\frac{d}{dt}\int_{\{H(\nabla u_\e)>t\}} f^2\,dx\bigg)^{1/2}=(-\mu'(t))^{1/2}\,\phi\big( \mu(t)\big)\quad\text{for a.e. $t>0$,}
\end{equation}
and, by \cite[Proposition 3.4]{cia11},  
\begin{equation}\label{pseudo2}
    \int_0^s \phi^*(r)^2\,dr\leq \int_0^s f^*(r)^2\,dr\quad\text{for $s\in (0,|\Omega|)$.}
\end{equation} 
Multiplying through \eqref{sempre:2} by $b_\e(t)$, and using \eqref{psedo1} and the monotonicity of $b_\e$ enable one to 
to deduce that
\begin{align}\label{sempre:3}
        (b_\e(t))^2\leq &\, c\,(-\mu'(t))\mu(t)^{-2/n'}\,\|H(\nabla u_\e)\|_{L^\infty(\Omega)}\,\bigg(\int_0^{\mu(t)} (f^*)^2(r)\,dr\bigg)
        \\
        & +c\,b_\e\big(\|H(\nabla u_\e)\|_{L^\infty(\Omega)} \big)\,\|H(\nabla u_\e)\|\,(-\mu'(t))\,\mu(t)^{-1/n'}\,\phi\big(\mu(t)\big)\nonumber
        \\
        &+ c\,(-\mu'(t))\,\mu(t)^{-1/n'}\,b_\e\big(\|H(\nabla u_\e)\|_{L^\infty(\Omega)} \big)^2\,\|H(\nabla u_\e)\|_{L^\infty(\Omega)}\,|\B|^{**}\big(c'\,\mu(t)^{1/n'}\big)\nonumber
    \end{align}
for a.e. $t\in \big[H(\nabla u_\e)^*\big(|\Omega|/2\big), \|H(\nabla u_\e)\|_{L^\infty(\Omega)}\big]$. 
Next, define the function $F_\e : [0, \infty) \to [0,\infty)$ as
\begin{equation*}
    F_\e(t)=\int_0^t b_\e(s)^2\,ds\quad \text{for $t\geq 0$,}
\end{equation*}
and observe that, thanks to \eqref{b2} and \eqref{appr:2},
\begin{equation}\label{b3}
    F_\e(t)\leq b_\e^2(t)\,t\leq c\,F_\e(t)\quad \text{for $t \geq 0$,}
\end{equation}
and for some constant $c=c(i_a,s_a)$.
Fix $t_0\in \big[H(\nabla u_\e)^*(|\Omega|/2), \|H(\nabla u_\e)\|_{L^\infty(\Omega)} \big]$ to be chosen later. Since $H(\nabla u_\e)$ is a Sobolev function, then $H(\nabla u_\e)^*$ is (locally) absolutely continuous, and $H(\nabla u_\e)^*(\mu(t))=t$ for $t>0$.
An integration of \eqref{sempre:3} and  the change of variables $r=\mu(t)$ thus yield:
\begin{align}\label{sempre:4}
        F_\e\big(H(\nabla u_\e)^*(s)\big)\leq &\,F_\e(t_0)+c\,\|H(\nabla u_\e)\|_{L^\infty(\Omega)}\,\int_s^{\mu(t_0)} r^{-2/n'}\int_0^r (f^*)^2(\rho)\,d\rho\,dr
        \\
        &+ c\,b_\e\big(\|H(\nabla u_\e)\|_{L^\infty(\Omega)} \big)\,\|H(\nabla u_\e)\|\,\int_s^{\mu(t_0)}\,r^{-1/n'}\,\phi(r)\,dr\nonumber
        \\
        &+ c\,b_\e\big(\|H(\nabla u_\e)\|_{L^\infty(\Omega)} \big)^2\,\|H(\nabla u_\e)\|_{L^\infty(\Omega)}\,\int_s^{\mu(t_0)}|\B|^{**}\big(c'\,r^{1/n'}\big)\,r^{-1/n'}\,dr\nonumber
        \\
        \leq &\,F_\e(t_0)+c\,\|H(\nabla u_\e)\|_{L^\infty(\Omega)}\,\int_0^{\mu(t_0)} r^{-2/n'}\int_0^r (f^*)^2(\rho)\,d\rho\,dr\nonumber
        \\
        &+ c\,b_\e\big(\|H(\nabla u_\e)\|_{L^\infty(\Omega)} \big)\,\|H(\nabla u_\e)\|\,\int_0^{\mu(t_0)}\,r^{-1/n'}\,\phi(r)\,dr\nonumber
        \\
        &+\hat{c}\,F_\e\big( \|H(\nabla u_\e)\|_{L^\infty(\Omega)}\big)\,\int_0^{\mu(t_0)^{1/n'}}|\B|^{**}(c'\,r)\,r^{\frac{1}{n-1}}\,\frac{dr}{r}\quad\text{for $s\in [0,\mu(t_0))$,}\nonumber
    \end{align}
    for some constants $c,c',\hat{c}$ depending on $n,i_a,s_a,\l,\L,d_\Omega,\mathcal{L}_\Omega$.
Note that the last inequality also relies on  \eqref{b3}. 
The assumption $\partial \Omega\in W^2L^{(n-1,1)}$ entails that
\begin{equation}
    \|\B\|_{L^{(n-1,1)}(\partial \Omega)}=\int_0^{\H^{n-1}(\partial \Omega)} |\B|^{**}(r)\,r^{\frac{1}{n-1}}\,\frac{dr}{r}<\infty\,.
\end{equation}
 Thus, the function $G: (0, |\Omega|) \to [0, \infty)$, defined as
\begin{equation}\label{G:s}
    G(s)=\int_0^{s^{1/n'}}|\B|^{**}(c'\,r)\,r^{\frac{1}{n-1}}\,\frac{dr}{r} \quad \text{for $s\in (0, |\Omega|)$},
\end{equation}
 is  strictly increasing  and $G(s)\to 0$ as $s\to 0^+$. Therefore, there exists $s_0=s_0(n,\Omega)$ such that
\begin{equation}\label{stima:Gs0}
    G(s_0)\leq \frac{1}{2\,\hat{c}}\quad\text{and}\quad 
 s_0\leq \frac{|\Omega|}{2}.
 \end{equation}  
By setting $s=0$ and
\begin{equation}\label{def:t0} 
    t_0=H(\nabla u_\e)^*(s_0)\,.
\end{equation}
 in \eqref{sempre:4},  we deduce that
\begin{align}\label{sempre:5}
    F_\e\big( \|H(\nabla u_\e)\|_{L^\infty(\Omega)}\big)\leq &\,2\,F_\e(t_0)+2c\,\|H(\nabla u_\e)\|_{L^\infty(\Omega)}\,\int_0^{\mu(t_0)} r^{-2/n'}\int_0^r (f^*)^2(\rho)\,d\rho\,dr
    \\
    &+2c\,b_\e\big(\|H(\nabla u_\e)\|_{L^\infty(\Omega)} \big)\,\|H(\nabla u_\e)\|\,\int_0^{\mu(t_0)}\,r^{-1/n'}\,\phi(r)\,dr\,.\nonumber
\end{align}
Here, we have made use of the fact that $\|H(\nabla u_\e)\|_{L^\infty(\Omega)}=H(\nabla u_\e)^*(0)$
 By the Hardy-Littlewood inequality, the monotonicity of $\phi^*$ and \eqref{pseudo2}, 
\begin{align}\label{di:mezzo}
    \int_0^{|\Omega|}r^{-1/n'}\phi(r)\,dr&\,\leq\int_0^{|\Omega|}r^{-1/n'}\,\phi^*(r)dr=\int_0^{|\Omega|}r^{-1/n'}\big(\phi^*(r)^2\big)^{1/2}\,dr
    \\
    &\,\leq \int_0^{|\Omega|}r^{-1/n'}\Big(\frac{1}{r}\int_0^r\phi^*(s)^2\,ds  \Big)^{1/2}\,dr\nonumber
     \leq \int_0^{|\Omega|}r^{-1/n'}\Big(\frac{1}{r}\int_0^r f^*(s)^2\,ds  \Big)^{1/2}\,dr\nonumber
    \\
    &\,=\int_0^{|\Omega|}r^{-1/n'} \big[(f^2)^{**}(r)\big]^{1/2}\,dr.\nonumber
    \end{align}
    Assume first that $n\geq 3$.
An application of the Hardy-type inequality \cite[Theorem 4.1 (ii)]{carro} (see also \cite[Theorem 10.3.12 (ii)]{pick}) yields:
\begin{equation*}
    \int_0^{|\Omega|}r^{-1/n'} \big[(f^2)^{**}(r)\big]^{1/2}\,dr\leq c\,\int_0^{|\Omega|}r^{-1/n'} \big[(f^2)^*(r)\big]^{1/2}\,dr=c\,\int_0^{|\Omega|}r^{-1/n'} f^*(r)\,dr=c\,\|f\|_{L^{n,1}(\Omega)}
\end{equation*}
for some constant $c=c(n)$,
whence, by \eqref{di:mezzo},
\begin{equation}\label{pseudo:4}
\int_0^{|\Omega|}r^{-1/n'}\phi(r)\,dr\leq c\,\|f\|_{L^{n,1}(\Omega)}\,.
\end{equation}
On the other hand, by Fubini's Theorem,
\begin{align*}
    \int_0^{|\Omega|} r^{-2/n'}\int_0^r (f^*)^2(\rho)\,d\rho\,dr& =\int_0^{|\Omega|}(f^*)^2(\rho)\int_\rho^{|\Omega|} r^{-2/n'}dr\,d\rho
    =\frac{1}{2/n'-1}\int_0^{|\Omega|}(f^*)^2(\rho)\,\rho^{1-2/n'}\,d\rho\,.
\end{align*}
Owing to \cite[Theorem 3.1 (i)]{carro} (see also \cite[Theorem 10.3.8 (i)]{pick}),   
\begin{equation*}
    \Big(\int_0^{|\Omega|}(f^*)^2(\rho)\,\rho^{1-2/n'}\,d\rho\Big)^{1/2}\leq c\,\int_0^{|\Omega|} f^*(\rho)\,\rho^{-1/n'}\,d\rho=\|f\|_{L^{n,1}(\Omega)}
\end{equation*}
for some constant $c=c(n)$.
Combining the last two inequalities tells us that
\begin{equation}\label{pseudo:5}
    \int_0^{|\Omega|} r^{-2/n'}\int_0^r (f^*)^2(\rho)\,d\rho\,dr\leq c\,\|f\|_{L^{n,1}(\Omega)}^2
\end{equation}
for some constant $c=c(n)$.
\\ Assume now that
$n=2$. The  inequality \eqref{di:mezzo} reads:
\begin{equation}\label{pseudo:6}
\int_0^{|\Omega|}r^{-1/2}\phi(r)\,dr\leq \int_0^{|\Omega|}r^{-1/2} \big[(f^2)^{**}(r)\big]^{1/2}\,dr=\|f\|_{ (L^2)^{(1,\frac{1}{2})}(\Omega)}\,.
\end{equation}
Moreover, by Fubini's Theorem
\begin{align*}
\int_0^{|\Omega}r^{-1}\,\int_0^r f^*(\rho)^2\,d\rho\,dr
    =\int_0^{|\Omega|}f^*(s)^2\,\int_s^{|\Omega|}\frac{dr}{r}ds=\int_0^{|\Omega|}(f^2)^*(s)\,\log\Big(\frac{|\Omega|}{s} \Big)\,ds.
\end{align*}
Owing to \cite[Theorem 10.3.17 (i)]{pick},  
\begin{equation*}
    \int_0^{|\Omega|}(f^2)^*(s)\,\log\Big(\frac{|\Omega|}{s} \Big)\,ds\leq c\,\bigg(\int_0^{|\Omega|}r^{-1/2} \big[(f^2)^{**}(r)\big]^{1/2}\,dr\bigg)^2=c\,\|f\|^2_{ (L^2)^{(1,\frac{1}{2})}(\Omega)}\,,
\end{equation*}
for some absolute constant $c$,
and hence
\begin{equation}\label{pseudo:7}
\int_0^{|\Omega}r^{-1}\,\int_0^r f^*(\rho)^2\,d\rho\,dr\leq c\,\|f\|^2_{ (L^2)^{(1,\frac{1}{2})}(\Omega)}.
\end{equation}
 The use of the inequalities \eqref{pseudo:4}--\eqref{pseudo:5} or \eqref{pseudo:6}--\eqref{pseudo:7}, according to whether $n\geq 3$ or $n=2$, in   \eqref{sempre:5} tells us that
\begin{align*}
        F_\e\big( \|H(\nabla u_\e)\|_{L^\infty(\Omega)}\big)\leq 2\,F_\e(t_0)&+c\,b_\e\big(\|H(\nabla u_\e)\|_{L^\infty(\Omega)} \big)\,\|H(\nabla u_\e)\|\,\|f\|_{X_n(\Omega)}
       +c\,\|H(\nabla u_\e)\|_{L^\infty(\Omega)}\,\|f\|^2_{ X_n(\Omega)}
\end{align*}
for some constant  $c=c( n,i_a,s_a,\l,\L,d_\Omega,\mathcal{L}_\Omega)$.
Dividing both sides of the latter inequality by $\|H(\nabla u_\e)\|_{L^\infty(\Omega)}$ and using \eqref{b3} and the monotonicity of $b_\e$ enable one to obtain:
\begin{equation}\label{sempre:6}
    b_\e\big( \|H(\nabla u_\e)\|_{L^\infty(\Omega)}\big)^2\leq 2\,b_\e(t_0)^2+c\,b_\e\big(\|H(\nabla u_\e)\|_{L^\infty(\Omega)} \big)\,\|f\|_{X_n(\Omega)}+c\,\|f\|^2_{X_n(\Omega)}\,.
\end{equation}
Hence, by Young's inequality, 
\begin{equation}\label{ci:siamo}
    b_\e\big( \|H(\nabla u_\e)\|_{L^\infty(\Omega)}\big)\leq c\,b_\e(t_0)+c\,\|f\|_{\mathbb{X}_{rhs}}
\end{equation}
for some constant $c=c(n,i_a,s_a,\l,\L,d_\Omega,\mathcal{L}_\Omega)$.
\\ Let $\beta_\e,\psi_\e:[0,\infty)\to [0,\infty)$ be the functions defined by
\begin{equation*}
\beta_\e(t)=b_\e(t)\,t\quad\text{and}\quad \psi_\e(t)=s\,b_\e^{-1}(t) \quad \text{for $t \geq 0$.}
\end{equation*}
Thanks to \eqref{appr:2} and  \eqref{b1},  
\begin{align}\label{beta:psi}
    \tfrac 1c\,\beta_\e(t)&\leq B_\e(t)\leq c\,\beta_\e(t)
    \\
    \tfrac 1c\,\psi_\e(s)&\leq \widetilde{B}_\e(s)\leq c\,\psi_\e(s)\,,\nonumber
\end{align}
for some constant $c=c(i_a,s_a)>1$.
\\
Hence, from \eqref{appr:2}, the embedding $X_n(\Omega) \to L^n(\Omega)$, 
%
and the energy estimate \eqref{stima:energiadir} with $B$ and $u$ replaced with $B_\e$ and $u_\e$ we deduce that
\begin{align}c\,\psi_\e(c\|f\|_{X_n(\Omega)}&)\geq \int_\Omega B_\e\big( H(\nabla u_\e)\big)\geq \int_{\{H(\nabla u_\e)>t_0\}} B_\e\big( H(\nabla u_\e)\big)\,dx
      \geq B_\e(t_0)\,\mu(t_0)\geq c'\,\beta_\e(t_0)\,s_0,\nonumber
\end{align}
for some positive constants $c=c(n,\l,\L,i_a,s_a,d_\Omega,\mathcal{L}_\Omega)$ and $c'=c'(i_a,s_a)$.
Observe that the first and the last inequalities hold thanks to \eqref{beta:psi} and \eqref{def:t0}, respectively. Therefore, owing to \eqref{beta:psi}, \eqref{appr:2}, and \eqref{Btildedelta2} with $\widetilde{B}_\e$ in place of $\widetilde{B}$,   we deduce that
\begin{equation}\label{eq:remark}
\beta_\e(t_0)\leq \frac{c_0}{s_0}\,\psi_\e(c_0\|f\|_{X_{n}})\leq \psi_\e(c_1\|f\|_{X_n(\Omega)})\,,
\end{equation}
for some positive constants $c_0(n,\l,\L,i_a,s_a,d_\Omega,\mathcal{L}_\Omega)$ and  $c_1=\big(n,\l,\L,i_a,s_a,d_\Omega,\mathcal{L}_\Omega,s_0(n,\Omega)\big)$.
\\
Since $b_\e\big( \beta_\e^{-1}(\psi_\e(s))\big)=s$ for $s\geq 0$,  the latter inequality implies that
$$
b_\e(t_0)\leq c_1\,\|f\|_{X_n(\Omega)},
$$
and hence, by
\eqref{ci:siamo},  
\begin{equation}\label{ci:siamoeps}
    b_\e\big(\|H(\nabla u_\e)\|_{L^\infty(\Omega)}\big)\leq
c_2\,\|f\|_{X_n(\Omega)}\,,
\end{equation}
where $c_2\big(n,\l,\L,i_a,s_a,d_\Omega,\mathcal{L}_\Omega,s_0(n,\Omega)\big)$.
Passing to the limit as  $\e\to 0^+$ in \eqref{ci:siamoeps} and making use of 
\eqref{bounds:H}, \eqref{cortona10}, and \eqref{conv:C1u} yield 
\begin{equation}\label{ci:siamoeps'}
    b\big(\|H(\nabla u)\|_{L^\infty(\Omega)}\big)\leq
c_2\,\|f\|_{X_n(\Omega)}.
\end{equation}

\smallskip
\par
\noindent \textit{Step 4.} Here, we keep the assumption \eqref{fcinf} in force and remove \eqref{deom:inf}.  By \cite[Theorem 1]{Antonini}, there exists a sequence of $C^\infty$-smooth domains $\{\Omega_m\}$ such that
\begin{equation}\label{conv:Omm}
    \Omega\subset\subset \Omega_m,\quad \lim_{m\to\infty}|\Omega_m\setminus\Omega|=0,\quad \lim_{m\to\infty}\mathrm{dist}_{\H}(\Omega_m,\Omega)=0,\quad d_{\Omega_m}\leq 2\,d_\Omega\,,
\end{equation}
for $m\in \N$, where $\mathrm{dist}_{\H}(\Omega_m,\Omega)$ denotes the Hausdorff distance between $\Omega_m$ and $\Omega$. Moreover, their Lipschitz characteristics $\mathcal{L}_{\Omega_m}=(L_{\Omega_m},R_{\Omega_m})$ satisfy:
\begin{equation}\label{lipca:Omm}
    L_{\Omega_m}\leq c(1+L_\Omega^2)\quad\text{and}\quad R_{\Omega_m}\geq \frac{R_\Omega}{c(1+L_\Omega^2)}\quad\text{for $m\in \N$,}
\end{equation}
for some constant $c=c(n)$.
Let $u_m\in W^{1,B}(\Omega_m)$ be the unique solution to the Dirichlet problem
\begin{equation}\label{eq:dirOmm}
    \begin{cases}
        -\mathrm{div}\big( \A(\nabla u_m)\big)=f\quad&\text{in $\Omega_m$}
        \\
        u_m=0\quad &\text{on $\partial \Omega_m$.}
    \end{cases}
\end{equation}
From  {\eqref{ci:siamoeps'}} of Step 3,  we have that
\begin{equation}\label{Lip:um}
   {  b} \big(\|\nabla u_m\|_{L^\infty(\Omega)}\big)\leq c_2\,\|f\|_{X_n(\Omega)}\,,
\end{equation}
for some constant $c_2=c_2(n,\l,\L,i_a,s_a,d_{\Omega_m},\mathcal{L}_{\Omega_m},s_0(n,\Omega_m))$. Here $s_0(n,\Omega_m)$ is the constant satisfying \eqref{stima:Gs0} with $\Omega$ replaced with $\Omega_m$, i.e., 
\begin{equation}\label{stima:Gsm0}
    G_m(s_0)\leq \frac{1}{2\hat{c}_m}\quad \text{and} \quad s_0\leq \frac{|\Omega_m|}{2}\,,
\end{equation}
where 
\begin{equation*}
    G_m(s)=\int_0^{s^{1/n'}}|\B_m|^{**}(c'\,r)\,r^{\frac{1}{n-1}}\frac{dr}{r}\,,
\end{equation*}
and $\hat{c}_m=\hat{c}_m(n,\l,\L,\mathcal{L}_{\Omega_m},d_{\Omega_m})$ is the constant appearing in \eqref{sempre:4} with $\Omega$ replaced with $\Omega_m$. Moreover, we have set $\B_m=\B_{\Omega_m}$,  the second fundamental form of $\partial \Omega_m$.
\\
We claim that \begin{equation}\label{goal:c2}
    c_2=c_2(n,\lambda,\Lambda,i_a,s_a,\Omega),
\end{equation}
namely, $c_2$ can be chosen so that it is independent of $m$.
To verify this claim, notice first of all that  $c_2$ depends on $d_{\Omega_m}$ and $L_{\Omega_m}$ only through an upper bound, and on $R_{\Omega_m}$ and $|\Omega_m|$ through a lower bound. Thus, owing to \eqref{conv:Omm} and \eqref{lipca:Omm}, we have $c_2=c_2(n,\lambda,\Lambda,i_a,s_a,d_\Omega,\mathcal{L}_\Omega,s_0(n,\Omega_m))$.
It remains to prove that one can choose $s_0=s_0(n,\Omega)$ which  makes \eqref{stima:Gsm0} true for all $m\in \N$.
\\
A close inspection of the proof of Step 3 reveals that the constant $\hat c_m$ in \eqref{stima:Gsm0}
depends on $d_{\Omega_m}$ and $L_{\Omega_m}$ via an upper bound, and on $R_{\Omega_m}$ and $|\Omega_m|$ via a lower bound. Therefore, by    \eqref{conv:Omm} and \eqref{lipca:Omm},
\begin{equation}\label{cc:hat}
    \hat{c}_m\leq \hat{c}(n,\l,\L,\mathcal{L}_{\Omega},d_{\Omega}).
\end{equation}
Assume, for a moment, that we know that
\begin{equation}\label{claim:Gs}
    G_m(s)\leq c\,G(c\,s)+c\,s^{1/n}\,\quad \text{for $s>0$,}
\end{equation}
for some constant $c=c(n,\mathcal{L}_\Omega,d_\Omega)$. 
Thus, by choosing $s_0=s_0(n,\Omega)$ such that
\begin{equation}\label{G:fin}
    c\,G(cs_0)+cs_0^{1/n}\leq \frac{1}{2\hat{c}}\quad \text{and} \quad s_0\leq \frac{|\Omega|}{2},
\end{equation}
 Equation \eqref{stima:Gsm0} follows from \eqref{claim:Gs}--\eqref{G:fin}, and the fact that $|\Omega|\leq |\Omega_m|$ by \eqref{conv:Omm}.
\\ With Equations \eqref{stima:Gsm0} and \eqref{lipca:Omm} at our disposal, the same arguments as in Step 3 imply that Equation \eqref{Lip:um} holds  with a constant $c_2=c_2(n,\l,\L,i_a,s_a,\Omega)$, and hence independent of $m$.
\\ 
Our task is now to prove Equation \eqref{claim:Gs}. Recall that by \cite[Theorem 1]{Antonini}, the boundaries $\partial\Omega_m$ share the same coordinate cylinders as $\partial \Omega$ -- see Figure 1.
\begin{figure}[ht]\label{fig}
\centering
\includegraphics[width=0.7\textwidth]{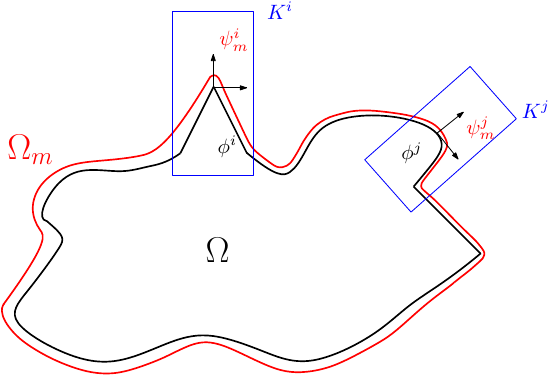}
\caption{}
\end{figure}
We denote 
 by  $\{K^i\}_{i=1}^N$ the family of  these coordinate cylinders, and by $\phi^i, \psi^i_m$ the functions locally describing the boundaries of $\Omega$ and $\Omega_m$, respectively. Hence, we have that
  \begin{equation}\label{de:subset}
  \partial \Omega\cup \partial \Omega_m\subset \bigcup_{i=1}^N K^i,
  \end{equation}
  and the cardinality $N$ of such coordinate cylinders satisfies  $N\leq c(n)\Big(\frac{d_\Omega}{R_\Omega}\Big)^n$.
\\
Furthermore, owing to \cite[Formula (6.48)]{Antonini},  we have that
\begin{equation}\label{psi:m}
    |\nabla^2\psi^i_m|\leq c\,\Big(\sum_{j\in\mathcal{I}_i} |\nabla^2\phi^j_m|\circ\C^{i,j}_m+1\Big)\quad \text{for $i=1,\dots, N$,}
\end{equation}
 where $\mathcal{I}_i$ is a finite set of indices, with cardinality $\mathrm{card}\,\mathcal{I}_i\leq N$. In Equation \eqref{psi:m}, the functions $\phi^j_m$  are defined as the convolution $\phi^j_m=\phi^j\ast \rho_m$ in $\R^{n-1}$, and $\C^{i,j}_m$ are a family of  bi-Lipschitz transformations, with Lipschitz constants depending on $n,d_\Omega,\mathcal{L}_\Omega$ -- see \cite[formula (6.43)]{Antonini}. This information coupled with \eqref{sub:add} yields
\begin{equation}\label{uuuh}
    |\nabla^2\psi^i_m|^{**}(r)\leq c\,
   \Big(\sum_{j\in 
   \mathcal{I}_i}|\nabla^2 \phi^j_m|^{**}(c\,r)+1\Big) \quad \text{for $r>0$,}
\end{equation}
for some constant $c=c(n,d_\Omega,\mathcal{L}_\Omega)$.
We denote by $\B|_{K_i},\B_m|_{K_i}$ the restrictions to $K^i$ of the second fundamental forms of $\partial\Omega$ and $\partial\Omega_m$, respectively. Then, owing to  \cite[Equation (2.10)]{Antonini}, \eqref{lipca:Omm} and basic properties of rearrangements, we have
\begin{align}\label{B:rearr}
   c(n,L_\Omega)\,|\nabla^2 \phi^i|^{**}\leq  |\B|_{K^i}|^{**}\leq |\nabla^2 \phi^i|^{**}
   \quad\text{and}\quad
    c(n,L_\Omega)\,|\nabla^2 \psi^i_m|^{**}\leq  |\B_m|_{K^i}|^{**}\leq |\nabla^2 \psi_m^i|^{**}
\end{align}
for  $i=1,\dots,N$ and $m\in \N$.
Equations  \eqref{de:subset}, \eqref{uuuh} and \eqref{B:rearr} entail that
\begin{align}\label{clao}
        G_m(s) & \leq c\,\sum_{i=1}^N\int_0^{s^{1/n'}}\big|\B_m|_{K^i}\big|^{**}(c\,r)\,r^{\frac{1}{n-1}}\frac{dr}{r}  \leq c'\,\sum_{i=1}^N\int_0^{s^{1/n'}}\big|\nabla^2\psi^i_m\big|^{**}(c'\,r)\,r^{\frac{1}{n-1}}\frac{dr}{r}
        \\
        &\leq c''\,\sum_{i=1}^N \sum_{j\in \mathcal{I}_i}\int_0^{s^{1/n'}}|\nabla^2 \phi_m^j|^{**}(c''r)r^{\frac{1}{n-1}}\frac{dr}{r}+c''\,\int_0^{s^{1/n'}}r^{\frac{1}{n-1}}\frac{dr}{r}\nonumber
        \\
        &=c''\,\sum_{i=1}^N \sum_{j\in \mathcal{I}_i}\int_0^{s^{1/n'}}\sum_{j\in \mathcal{I}_i}|\nabla^2 \phi_m^j|^{**}(c''r)r^{\frac{1}{n-1}}\frac{dr}{r}+c'''\,s^{1/n}\nonumber \quad \text{for $s>0$,}
\end{align}
for suitable constants $c,c',c'', c'''$ depending on $n,\mathcal{L}_\Omega,d_\Omega$.
\\ The property \eqref{convol:star1}
%
guarantees that
\begin{equation}\label{convol:star}
   |\nabla ^2\phi_m^j|^{**}(r)=  |\nabla ^2\phi^j\ast\rho_m|^{**}(r)\leq c(n) (|\nabla ^2\phi^j|\ast\rho_m)^{**}(r)\leq
   c'(n)\,|\nabla^2\phi^j|^{**}(r)\quad\text{for $r>0$,}
\end{equation}
for $j=1, \dots , N$, and $m\in \N$. Hence,
from  \eqref{B:rearr} -- \eqref{convol:star} we deduce that
\begin{align}
    G_m(s)& \leq c\,\sum_{i=1}^N \sum_{j\in \mathcal{I}_i}\int_0^{s^{1/n'}}|\nabla^2 \phi^j|^{**}(cr)r^{\frac{1}{n-1}}\frac{dr}{r}+c\,s^{1/n}
     \leq c'\,\sum_{i=1}^N \sum_{j\in \mathcal{I}_i}\int_0^{s^{1/n'}} |\B|_{K^j}|^{**}(c'r)r^{\frac{1}{n-1}}\frac{dr}{r}+c\,s^{1/n}\nonumber
    \\
    &\leq c''\,N^2\,\int_0^{s^{1/n'}}|\B|^{**}(c'r)r^{\frac{1}{n-1}}\frac{dr}{r}+C\,s^{1/n}\nonumber 
    =c''\,G(c''\,s)+c\,s^{1/n}\nonumber \qquad \text{for $s>0$,}
\end{align}
for some constants $c, c', c''$ depending on $n,d_\Omega,\mathcal{L}_\Omega$. The inequality \eqref{claim:Gs} is thus established.
\\
 Now let $B_R$ be a ball such that  $\Omega\subset \Omega_m\subset B_R$, and consider the  extension of $u_m$ (still denoted  by $u_m$) by $0$ on $B_R\setminus \Omega$. 
Thanks to \eqref{Lip:um}, the sequence $\{u_m\}$ is bounded in $W^{1,\infty}(B_R)$ and, up to a subsequence, we have
\begin{equation}\label{pap_m2}
    u_m\to v\quad\text{weakly-* in $W^{1,\infty}(B_R)$\quad  and}\quad u_m\to v\quad\text{a.e. in $B_R$,}
\end{equation}
for some function $v\in W^{1,\infty}(B_R)$.
\\
Also, by \eqref{mon:updown}, we can apply \cite[Theorem 1.7]{lieb91} to infer that, for any open set $\Omega'\Subset \Omega$, one has  $\|u_m\|_{C^{1,\theta}(\Omega')}\leq c$ for some constants $\theta\in (0,1)$ and $c>0$ independent of $m$. Therefore, up to a subsequence,
\begin{equation}\label{pap_mm2}
    u_m\to v\quad\text{in $C^1_{loc}(\Omega)$.}
\end{equation}
It remains to show that $v=u$. Let $x\in B_R\setminus \Omega$. Then, thanks to \eqref{conv:Omm}, $x\in B_R\setminus \Omega_m$ if $m$ is large enough. Since $u_m=0$ in $B_R\setminus \Omega_m$, by \eqref{pap_m2} it follows that $v(x)=0$ in $B_R\setminus \Omega$, and in particular $v\in W^{1,\infty}_0(\Omega)$.
Consider a test function $\varphi\in C^\infty_c(\Omega)$ in the weak formulation \eqref{eq:dirOmm}, and
\\
Letting $m\to \infty$ in the weak formulation of the problem \eqref{eq:dirOmm}
shows, via \eqref{pap_m2} -- \eqref{pap_mm2}, that $v$ is a solution to \eqref{eq:dir2}. Hence, $u=v$. The lower semicontinuity of the norm and \eqref{Lip:um} yield
\begin{equation*}
   {  b }\big(\|\nabla u\|_{L^\infty(\Omega)}\big)\leq \liminf_{m\to\infty}  {  b }\big(\|\nabla u_m\|_{L^\infty(\Omega)}\big)\leq c_2\,\|f\|_{X_n(\Omega)}\,.
\end{equation*}

\smallskip
\noindent \textit{Step 5.} Here, we remove the assumption \eqref{fcinf}. For any $k\in \N$, let
$$\overline{f}_k(x) = 
\begin{cases}
    f(x)\quad & \mathrm{dist}(x,\partial \Omega)>\frac{1}{k}

    \\
    0\quad &\mathrm{dist}(x,\partial \Omega)\leq\frac{1}{k}\,,
\end{cases}
$$
and let 
$$
f_k=\overline{f}_k\ast\rho_{1/(2k)}\,,
$$
where $\rho_{1/(2k)}$ is the standard, radially symmetric convolution kernel of parameter $1/(2k)$.
It follows that $f_k\in C^\infty_c(\Omega)$ and  
\begin{equation}\label{fk:conv}
    f_k\to f\quad\text{in $L^n(\Omega)$.}
\end{equation}
Moreover, thanks to the  properties \eqref{convol:star1} and \eqref{oct100},  
\begin{equation}\label{fk:conv1}
    \|f_k\|_{X_n(\Omega)}\leq \|f\|_{X_n(\Omega)}\,.
\end{equation}
Let $u_k\in W^{1,B}(\Omega)$ be the solution to 
\begin{equation}\label{eq:diruk}
\begin{cases}
        -\mathrm{div}\big(\A(\nabla u_k) \big)=f_k &\quad\text{in } \Omega
        \\
        u_k=0 &\quad\text{on }\partial \Omega \,.
\end{cases}
\end{equation}
From the previous step and \eqref{fk:conv}-\eqref{fk:conv1}, we have that
\begin{equation}\label{aaah}
    {  b}\big(\|\nabla u_k\|_{L^\infty(\Omega)}\big)\leq c_0\|f\|_{X_n(\Omega)}
\end{equation}
for some constant $c_0=c_0(n,\lambda,\Lambda,i_a,s_a,\Omega)$, and hence independent of $k$. Therefore, up to  subsequences,
\begin{equation}\label{uk:debole}
    u_k\to v\quad\text{weakly-* in $W^{1,\infty}_0(\Omega)$.}
\end{equation}
  We  claim that
\begin{equation}\label{gr:conv1}
    \nabla u_k\to \nabla v\quad\text{a.e. in $\Omega$,}
\end{equation}
 again up to subsequences.
Set $\tau_0={  b^{-1}}\big(c_0\,\|f\|_{X_n(\Omega)}\big)$. Given $t>0$, owing to \eqref{aaah} we have $|\{|\nabla u_k|>\tau_0\}|=0$ for all $k\in \N$. Consequently, 
\begin{equation}
\begin{split}
    |\{ &|\nabla u_k-\nabla u_l|>t\}= |\{|\nabla u_k-\nabla u_l|>t, \,|\nabla u_k|\leq \tau_0,\,|\nabla u_l|\leq \tau_0\}|.
    \end{split}
\end{equation}
for $k,l\in \N$. Set
\begin{equation*}
    \vartheta=\inf\{[\A(\xi)-\A(\eta)]\cdot (\xi-\eta):\,|\xi-\tau|>t,\,  |\xi|\leq \tau_0,\,|\eta|\leq \tau_0 \},
\end{equation*}
and observe that $\vartheta>0$ by \eqref{monot:A}. Making use of $u_k-u_l$ as a test function in the weak formulation of \eqref{eq:diruk} and of its analogue with $k$ replaced with $l$ and subtracting the resultant equations tell us that
\begin{align}\label{dim:ae}
    \vartheta\,&|\{|\nabla u_k-\nabla u_l|>t, \,|\nabla u_k|\leq \tau_0,\,|\nabla u_l|\leq \tau_0\}|
    \\
    &\leq \int_\Omega \big(\A(\nabla u_k)-\A(\nabla u_l)\big)\cdot (\nabla u_k-\nabla u_l)\,dx \nonumber
    \\
    &=\int_\Omega (f_k-f_l)\,(u_k-u_l)\,dx\leq \|u_k-u_l\|_{L^{n'}(\Omega)}\,\|f_k-f_l\|_{L^n(\Omega)} \nonumber
    \\
    &\leq c\,\|\nabla u_k-\nabla u_l\|_{L^1(\Omega)}\,\|f_k-f_l\|_{L^n(\Omega)} \nonumber
    \\
    &\leq 2\,c'\,\tau_0\,\|f_k-f_l\|_{L^n(\Omega)} \nonumber
\end{align}
for some constants $c,c'=c,c'(n,\Omega)$. Equations \eqref{fk:conv} and \eqref{dim:ae} imply that $\{\nabla u_k\}$ is a Cauchy sequence in measure. From this piece of information and \eqref{uk:debole} we deduce \eqref{gr:conv1}.
Passing to the limit as $k\to\infty$ in the weak formulation of the problem \eqref{eq:diruk} and making use of  Equations \eqref{fk:conv} and \eqref{gr:conv1} ensure that $v$ solves the Dirichlet problem \eqref{eq:dir2}.  Hence $v=u$, and the  estimate \eqref{stima:dirichlet} follows from \eqref{aaah}, \eqref{Btildedelta2} and \eqref{b1}, via  semicontinuity.
\end{proof}

\section{Proof of Theorem \ref{thm:dirneu}, Neumann problems}

The outline of the proof of 
Theorem \ref{thm:dirneu} for solutions to Neumann problems is the same as for Dirichlet problems. The necessary modifications are described in this section.

\medskip \par
\noindent
\begin{proof}[Proof of  Theorem \ref{thm:dirneu},  Neumann problems] 
\textit{Step 1.}  Assume that the function $f$ and the domain $\Omega$ satisfy the assumptions \eqref{fcinf} and \eqref{deom:inf}, respectively. Let 
$\mathcal A_\e$ be defined as in \eqref{a_e} for $\e \in (0,1)$ and let 
$u_\e\in W^{1,2}_\perp(\Omega)$ be the weak solution to the Neumann problem
\begin{equation}\label{eq:ueneum}
    \begin{cases}
        -\mathrm{div}\big( \A_\e(\nabla u_\e)\big)=f\quad& \text{in $\Omega$}
        \\
        \A_\e(\nabla u_\e)\cdot \nu=0\quad& \text{on $\partial\Omega$.}
    \end{cases}
\end{equation}
We claim that 
\begin{equation}\label{reg:ueneum}
    u_\e\in C^{1,\beta}(\overline{\Omega})\cap W^{2,2}(\Omega)\,,
\end{equation}
for some $\beta\in (0,1)$ independent on $\e$, and
\begin{equation}\label{conv:C1neu}
   \nabla u_\e\to \nabla u\quad\text{in $C^0(\overline{\Omega})$}\,.
\end{equation}
First, an application of \cite[Theorem 3.1 (a)]{cia90} ensures that, after normalizing   $u_\e$  by suitably additive constants depending on $\e$, we have that
\begin{equation}
    \|u_\e\|_{L^\infty(\Omega)}\leq c\,,
\end{equation}
for some constant $c$ independent on $\e$. Then, from \cite[Theorem 1.7 and subsequent remarks]{lieb91} (see also \cite[Theorem 2]{lieb88} and \cite{Antonini1}) we deduce that
\begin{equation}
    \|u_\e\|_{C^{1,\beta}(\overline{\Omega})}\leq c\,,
\end{equation}
for some constants $c>0$ and $\beta\in (0,1)$ independent on $\e$. By the Ascoli-Arzel\`a   theorem, given any sequence $\e_k\to 0^+$, there exist a subsequence, still denoted by  $\e_k$ and a function $v \in C^1(\overline{\Omega})$  such that $u_\e\to v$ in $C^1(\overline{\Omega})$. Passing to the limit  as $k \to \infty$ in the weak formulation of \eqref{eq:ueneum} with $\e$ replaced with $\e_k$ tells us that $v$ is solution to the Neumann problem \eqref{eq:neu2}.  Hence, $v=u$ up to an additive constant. Since the same conclusion is independent of the original sequence $\e_k$, Equation \eqref{conv:C1neu} follows.
\\ Thanks to \eqref{appr:1} and \eqref{pr:ae},  an analogous difference quotient argument as in the proof of \cite[Theorem 8.2]{ben:fre} adapted to (homogeneous) Neumann boundary condition, ensures that $u_\e\in W^{2,2}(\Omega)$.  Equation \eqref{reg:ueneum} is thus established.
\smallskip
\par \noindent
\textit{Step 2.} 
Let $f$, $\Omega$, and $u_\e$ be as in Step 1.
This step is devoted to a version of the inequality \eqref{sempre}
for solutions to Neumann problems. The relevant inequality tells us that
there exists a positive constant $c=c(n,i_a,s_a,\l,\L)$ such that
\begin{align}\label{sempre:neu}c \int_{\{H(\nabla u_\e)=t\}}b_\e(t) &\,|\nabla H(\nabla u_\e)|\, d\H^{n-1}\leq  \int_{\{H(\nabla u_\e)
    >t\}}\frac{f^2}{a_\e\big( H(\nabla u_\e)\big)}\,dx+\,t\int_{\{H(\nabla u_\e)=t\}}|f|\,d\H^{n-1}
    \\
    &-\int_{\partial \Omega\cap \{H(\nabla u_\e)>t\}}a_\e\big( H(\nabla u_\e)\big)\,\B\Big(\big[\tfrac{1}{2}\nabla_\xi H^2(\nabla u_\e)\big]_T\,,\big[\tfrac{1}{2}\nabla_\xi H^2(\nabla u_\e)\big]_T \Big)\,d\H^{n-1}\nonumber
    \end{align}   
    for a.e. $t>0$.
\\ The inequality \eqref{sempre:neu} can be established as follows.
By Stein's extension theorem \cite[Theorem 13.17]{leoni}, the function $u_\e$ can be continued to the entire $\rn$ to a function, still denoted by $u_\e$, such that $u_\e\in W_c^{2,2}(\R^n)\cap C_c^{1,\theta}(\R^n)$.
Thus, there exists a sequence of functions $\{u_{\e,k}\}\subset C_c^\infty(\R^n)$, denoted by $\{u_{k}\}$ for simplicity in what follows,  such that
\begin{equation}\label{pppp}
    u_k\xrightarrow{k\to\infty} u_\e \quad\text{in $W^{2,2}(\R^n)\cap C^{1}(\R^n)$.}
\end{equation}
Hence, an application of \cite[Proposition 2.6]{musina} enables one to deduce that
\begin{equation}\label{pppp1}
    \A_\e(\nabla u_k)\xrightarrow{k\to\infty} \A_\e(\nabla u_\e)\,,\quad \tfrac{1}{2}   \nabla_\xi H^2(\nabla u_k)\xrightarrow{k\to\infty} \tfrac{1}{2}\nabla_\xi H^2(\nabla u_\e)\quad\text{in $W^{1,2}(\R^n)\cap C^0(\R^n)$.} 
\end{equation}
Next,  owing to \eqref{gen21}, \eqref{gen24:neu} and \eqref{def:Ae}, 
\begin{align}\label{cont:4neu}
     &\mathrm{div} [\tfrac{1}{2}\nabla_\xi H^2(\nabla u_k)] \,\A_\e(\nabla u_k)\cdot\nu-  \nabla [\tfrac{1}{2}\nabla_\xi H^2(\nabla u_k)]\, [\A_\e(\nabla u_k)] \cdot \nu
\\
    &= \mathrm{div}_T [\tfrac{1}{2}\nabla_\xi H^2(\nabla u_k)] \,\A_\e(\nabla u_k)\cdot\nu  -  \nabla_T [\tfrac{1}{2}\nabla_\xi H^2(\nabla u_k)]\, [\A_\e(\nabla u_k)]_T \cdot \nu\nonumber
    \\
    &= \mathrm{div}_T [\tfrac{1}{2}\nabla_\xi H^2(\nabla u_k)] \,\A_\e(\nabla u_k)\cdot\nu-\nabla_T\Big( \tfrac{1}{2}\nabla_\xi H^2(\nabla u_k)\cdot\nu\Big)\cdot \big[\A_\e(\nabla u_k)\big]_T\nonumber
    \\
    &\hspace{3cm}+a_\e\big( H(\nabla u_k)\big)\,\B\Big(\big[\tfrac{1}{2}\nabla_\xi H^2(\nabla u_k)\big]_T,\big[\tfrac{1}{2}\nabla_\xi H^2(\nabla u_k)\big]_T \Big)\quad \text{on $\partial \Omega$.}\nonumber
\end{align}
Computations analogous to  \eqref{tr1}-\eqref{cont:3} and \eqref{temp:uk}-\eqref{equa:int}, with just a replacement of the identity \eqref{cont:4} with \eqref{cont:4neu}   in dealing with the integral over $\partial \Omega\cap  \{H(\nabla u_k)>t\}$, enable one to deduce that
\begin{align}\label{equa:intneu}
         \int_{\{t<H(\nabla u_k)<t+h\}}&\,b_\e\big(H(\nabla u_k)\big)\, |\nabla H(\nabla u_k)|^2\, dx 
         \\
         \leq &\,c\,\int_t^{t+h}       \int_{\{H(\nabla u_k)>s\}} \frac{f_k^2}{a_\e\big( H(\nabla u_k)\big)}\,dx\,ds\nonumber
        \\
        &+c\,\int_{\{t<H(\nabla u_k)<t+h\}} H(\nabla u_k) |f_k|\,|\nabla H(\nabla u_k)|\,dx\nonumber
        \\
&+c\,\int_t^{t+h}\int_{\partial \Omega\cap \{H(\nabla u_k)>s\}}\big(\mathrm{div}_T [\tfrac{1}{2}\nabla_\xi H^2(\nabla u_k)]\big)\,\A_\e(\nabla u_k)\cdot\nu\,d\H^{n-1}ds\nonumber
\\
&-c\,\,\int_t^{t+h}\int_{\partial \Omega\cap \{H(\nabla u_k)>s\}}\nabla_T\Big( \tfrac{1}{2}\nabla_\xi H^2(\nabla u_k)\cdot\nu\Big)\cdot \big[\A_\e(\nabla u_k)\big]_T\,d\H^{n-1}ds\nonumber
\\
    &-c\,\int_t^{t+h}\int_{\partial \Omega\cap \{H(\nabla u_k)>s\}}a_\e\big( H(\nabla u_k)\big)\,\B\Big(\big[\tfrac{1}{2}\nabla_\xi H^2(\nabla u_k)\big]_T\,,\big[\tfrac{1}{2}\nabla_\xi H^2(\nabla u_k)\big]_T \Big)\,d\H^{n-1}\,ds\nonumber
\end{align}
for $t, h >0$.
\\
Now, denote by $W^{\frac{1}{2},2}(\partial \Omega)$ the fractional Sobolev space on $\partial \Omega$, and by $W^{-\frac{1}{2},2}(\partial \Omega)$ its dual space. Then, the trace operator
\begin{equation*}
    \mathrm{Tr}:W^{1,2}(\Omega)\to W^{\frac{1}{2},2}(\partial \Omega)\,,
\end{equation*}
is a bounded linear operator. The operator
\begin{equation}\label{nT:cont}
\nabla_T: W^{\frac{1}{2},2}(\partial \Omega)\to W^{-\frac{1}{2},2}(\partial \Omega)\,,
\end{equation}
which associates to a smooth function $v$ its tangential gradient $\nabla_T v$ on $\partial \Omega$ is also linear and bounded. This property \eqref{nT:cont} is established in \cite[Lemma 1.4.1.3]{gris} in the case when $\Omega$ is the half space. The general case of a bounded open set  $\Omega$  with a smooth boundary follows via a standard covering and flattening argument.
\\ 
Therefore, owing to the boundary condition $\A_\e(\nabla u_\e)\cdot \nu=\tfrac{1}{2}\nabla_\xi H^2(\nabla u_\e)\cdot \nu=0$ on $\partial \Omega$ and Equation \eqref{pppp1}, we have that
\begin{align}\label{pap:m}
\bigg|\int_t^{t+h}\int_{\partial \Omega\cap \{H(\nabla u_k)>s\}}& \big(\mathrm{div}_T [\tfrac{1}{2}\nabla_\xi H^2(\nabla u_k)]\big)\,\A_\e(\nabla u_k)\cdot\nu\,d\H^{n-1}ds\bigg|
    \\
    &\leq h\,\|\nabla_T[\tfrac{1}{2}\nabla_\xi H^2(\nabla u_k)]\|_{W^{-\frac{1}{2},2}(\partial \Omega)}\,\| \A_\e(\nabla u_k)\cdot\nu\|_{W^{\frac{1}{2},2}(\partial \Omega)}\nonumber
    \\
    &\leq c\,h\,\|\tfrac{1}{2}\nabla_\xi H^2(\nabla u_k)\|_{W^{1,2}(\Omega)}\,\| \A_\e(\nabla u_k)\cdot\nu\|_{W^{\frac{1}{2},2}(\partial \Omega)}\nonumber
    \\
    &\xrightarrow{k\to\infty} c\,h\,\|\tfrac{1}{2}\nabla_\xi H^2(\nabla u_\e)\|_{W^{1,2}(\Omega)}\,\| \A_\e(\nabla u_\e)\cdot\nu\|_{W^{\frac{1}{2},2}(\partial \Omega)}=0\,,\nonumber
    \end{align}
and similarly
\begin{align}
       \bigg| \int_t^{t+h}\int_{\partial \Omega\cap \{H(\nabla u_k)>s\}}&\nabla_T\Big( \tfrac{1}{2}\nabla_\xi H^2(\nabla u_k)\cdot\nu\Big)\cdot \big[\A_\e(\nabla u_k)\big]_T\,d\H^{n-1}ds\bigg|
       \\
       &\leq h\,\big|\big|\nabla_T\big(\tfrac{1}{2}\nabla_\xi H^2(\nabla u_k)\cdot \nu\big)\big|\big|_{W^{-\frac{1}{2},2}(\partial \Omega)}\,\|\A_\e(\nabla u_k)\|_{W^{\frac{1}{2},2}(\partial \Omega)}\nonumber
       \\
       &\leq c\,h\,\|\tfrac{1}{2}\nabla_\xi H^2(\nabla u_k)\cdot \nu\|_{W^{\frac{1}{2},2}(\partial \Omega)}\,\|\A_\e(\nabla u_k)\|_{W^{\frac{1}{2},2}(\partial \Omega)}\nonumber
       \\
       &\xrightarrow{k\to\infty}c\,h\,\|\tfrac{1}{2}\nabla_\xi H^2(\nabla u_\e)\cdot \nu\|_{W^{\frac{1}{2},2}(\partial \Omega)}\,\|\A_\e(\nabla u_\e)\|_{W^{1,2}(\Omega)}=0\,.\nonumber
    \end{align}
On the other hand, from \eqref{pppp1} we infer that
\begin{align}\label{pap:m1}
\lim_{k\to\infty}&\int_t^{t+h}\int_{\partial \Omega\cap \{H(\nabla u_k)>s\}}a_\e\big( H(\nabla u_k)\big)\,\B\Big(\big[\tfrac{1}{2}\nabla_\xi H^2(\nabla u_k)\big]_T\,,\big[\tfrac{1}{2}\nabla_\xi H^2(\nabla u_k)\big]_T \Big)\,d\H^{n-1}\,ds
    \\
    &=\int_t^{t+h}\int_{\partial \Omega\cap \{H(\nabla u_\e)>s\}}a_\e\big( H(\nabla u_\e)\big)\,\B\Big(\big[\tfrac{1}{2}\nabla_\xi H^2(\nabla u_\e)\big]_T\,,\big[\tfrac{1}{2}\nabla_\xi H^2(\nabla u_\e)\big]_T \Big)\,d\H^{n-1}\,ds\,.\nonumber
    \end{align}
One can now make use of Equations \eqref{pap:m} -- \eqref{pap:m1} in passing to the limit as $k\to \infty$ in the inequality \eqref{equa:intneu}, and argue as in the proof of 
\eqref{sempre} to establish the inequality
%
\eqref{sempre:neu}. 

\smallskip
\par
\noindent\textit{Step 3.} Having the inequality \eqref{sempre:neu} at disposal, an argument  along the  same lines as in Step 3 of the proof for Dirichlet problems yields  the inequality \eqref{stima:dirichlet} under the assumptions \eqref{fcinf} and \eqref{deom:inf}.
One has just to replace the inequality \eqref{sempre} with
\eqref{sempre:neu}, and note that,  thanks to \eqref{bound:nabH1}, 
\begin{equation*}
    \bigg|a_\e\big( H(\nabla u_\e)\big)\,\B\Big(\big[\tfrac{1}{2}\nabla_\xi H^2(\nabla u_\e)\big]_T\,,\big[\tfrac{1}{2}\nabla_\xi H^2(\nabla u_\e)\big]_T \Big)\bigg|\leq c\, b_\e\big(H(\nabla u_\e)\big)\,H(\nabla u_\e)\,|\B|\quad \text{on $\partial \Omega$,}
\end{equation*}
for some constant $c=c(n,\l,\L)$. 

\noindent \textit{Steps 4 and 5.} 
The additional assumptions  \eqref{deom:inf} and \eqref{fcinf}
can be removed via arguments completely analogous to those employed in Step 4 and Step 5 for Dirichlet problems.
\end{proof}

\section{Convex domains: proof of Theorem \ref{thm:conv}}
The proof of Theorem \ref{thm:conv} parallels that of Theorem \ref{thm:dirneu}. It is fact simpler, as the convexity of the domain $\Omega$ ensures that the integrals over the boundary in the inequalities \eqref{sempre} and \eqref{sempre:neu} are nonnnegative, and hence can just be disregarded in the subsequent steps.

\begin{proof}[Proof of Theorem \ref{thm:conv}, sketched]
\noindent\textit{Step 1.} This step agrees with Step 1 in the proof of Theorem \ref{thm:dirneu}.
\smallskip\par
\noindent\textit{Step 2.} The convexity assumption on $\partial \Omega$ plays a crucial role here. Indeed, thanks to  \eqref{feb315}, one has that
\begin{equation*}
    \int_{\partial \Omega\cap \{H(\nabla u_\e)>t\}}b_\e\big(H(\nabla u_\e)\big)\,H(\nu)\,H(\nabla u_\e)\,\mathrm{tr}\,\B^H\,d\H^{n-1}\geq 0
\end{equation*}
and
\begin{equation*}
    \int_{\partial \Omega\cap \{H(\nabla u_\e)>t\}}a_\e\big( H(\nabla u_\e)\big)\,\B\Big(\big[\tfrac{1}{2}\nabla_\xi H^2(\nabla u_\e)\big]_T\,,\big[\tfrac{1}{2}\nabla_\xi H^2(\nabla u_\e)\big]_T \Big)\,d\H^{n-1}\geq 0
\end{equation*}
for a.e. $t>0$.
Therefore,  the inequalities \eqref{sempre} and \eqref{sempre:neu} can be replaced with 
\begin{align}
\label{sempreconv}
c\int_{\{H(\nabla u_\e)=t\}}b_\e(t)\,|\nabla H(\nabla u_\e)|\, d\H^{n-1}\leq &\int_{\{H(\nabla u_\e)
    >t\}}\frac{f^2}{a_\e\big( H(\nabla u_\e)\big)}\,dx+\,t\,\int_{\{H(\nabla u_\e)=t\}}|f|\,d\H^{n-1}
\end{align}   
for a.e. $t>0$, for both Dirichlet and Neumann problems.
\smallskip\par
\noindent\textit{Step 3.} This step is analogous to and simpler than the corresponding step in the proof of Theorem \ref{thm:dirneu}, since the right-hand side of \eqref{sempreconv} does not contain integrals over $\partial \Omega$.
In particular, the function $G$ in \eqref{G:s} does not come into play, and hence one can choose $s_0=|\Omega|/2$ in \eqref{stima:Gs0}. Therefore the constant $c_0$ in \eqref{stima:convex} only depends on the quantities specified in the statement.

\smallskip\par
\noindent\textit{Step 4.} This step is devoted to removing the assumption \eqref{deom:inf}. The set  $\Omega$ can be approximated by  a sequence of smooth convex domains $\Omega_m$ satisfying \eqref{conv:Omm} -- \eqref{lipca:Omm}.
This can be done, for instance, as in \cite[equation (7.73)]{accfm} and  \cite[Section 4.1]{BDMS}.
Then one may proceed as in the proof of \ref{thm:dirneu}, starting from \eqref{pap_m2} onward, since the  constant $c_0$ of Step 3 explicitly depends on $|\Omega|,\mathcal{L}_\Omega, d_\Omega$.  

\smallskip\par
\noindent\textit{Step 5.} This step agrees with its analogue in Theorem \ref{thm:dirneu}.

\end{proof}

\section{Natural growth in the gradient: proof of Theorems \ref{thm:example1} and \ref{thm:example} }\label{sec:naturalgrowth}

\begin{proof}[Proof of Theorems \ref{thm:example1} and \ref{thm:example}]
    The idea is to reduce the problems \eqref{Ex:dir}
and \eqref{Ex:neu} to problems of the form \eqref{eq:dir2} and \eqref{eq:neu2}, respectively, 
via the Kazdan-Kramer transformation, in the same spirit as \cite{figuer}. 
\\ 
Let $u$ be a solution to either of problems \eqref{eq:dir2} and \eqref{eq:neu2}.
Define the function $\psi : [0, \infty) \to [0, \infty)$ as
\begin{equation}
    \psi (t)=\int_0^t e^{\frac{\kappa}{p-1}s}\,d \tau \quad \text{for $s\geq 0$.}
\end{equation}
Set
 $$v=\psi(u).$$
Since $u\in W^{1,p}(\Omega)\cap L^\infty(\Omega)$ and $\psi$ is  Lipschitz continuous on bounded subsets of $[0, \infty)$, by the chain rule for Sobolev functions   $v\in W^{1,p}(\Omega)$. Moreover, if $u$ solves  \eqref{eq:dir2}, then $u\in W^{1,p}_0(\Omega)$, and hence 
$v\in W^{1,p}_0(\Omega)$ as well. 
Assume, instead, that $u$ solves \eqref{eq:neu2}. Since  
\begin{equation}
    \nabla v=\psi'(u)\,\nabla u= e^{\frac{\kappa}{p-1}u}\,\nabla u\,,
\end{equation}
and the function $\A$ is
$(p-1)$-homogeneous,
one has
\begin{equation*}
    \A(\nabla v)=e^{\kappa\,u}\,\A(\nabla u)\,.
\end{equation*}
Therefore, 
the function $v$  satisfies the co-normal Neumann boundary condition $\A(\nabla v)\cdot \nu=0$. Let us point out that the formal computation above can be justified via a parallel argument in the weak formulation of the  problem \eqref{eq:neu2}. 
\\ The same remark applies to the following computation of $ \Delta_p^H v$, which, in particular, rests upon the $(p-2)$-homogeneity of the function $\nabla_\xi \A$ and the $1$-homogeneity of $H$:
\begin{align*}
    \Delta_p^H v & =\mathrm{tr}\Big(\nabla_\xi\A(\nabla v)\,\nabla^2v\Big)=[\psi'(u)]^{p-2}\,\Big\{\nabla_\xi\A(\nabla u)\,\big[\psi'(u)\,\nabla ^2 u+\psi''(u)\nabla u\otimes \nabla u \big] \Big\}
    \\
    &=[\psi'(u)]^{p-1}\Big\{\mathrm{tr}\big(\nabla_\xi\A(\nabla u)\,\nabla^2 u\big)+\frac{\psi''(u)}{\psi'(u)}\,\tfrac{1}{p}\nabla_\xi^2H^p(\nabla u)\,\nabla u\cdot \nabla u \Big\}
    \\
    &= [\psi'(u)]^{p-1}\Big\{\Delta_p^H u+(p-1)\,\frac{\psi''(u)}{\psi'(u)}\,H^p(\nabla u)  \Big\}.
\end{align*}
%
Consequently, $v$ is a weak solution to the equation
\begin{equation}\label{eqv}
    -\Delta_p^H v=\widehat f\quad\text{in $\Omega$\,,}
\end{equation}
coupled with homogeneous Dirichlet or Neumann boundary conditions on $\partial \Omega$, where $\widehat f: \Omega \to \mathbb R$ is the function given by
$$\widehat f (x)= e^{\frac{\kappa}{p-1}\,u}\Big\{
 g(x,u,\nabla u)-\kappa\, H^p(\nabla u)\Big\} \quad \text{for $x \in \Omega$}.$$
   Thanks to the assumption \eqref{g}, it satisfies the estimate: $$|\widehat f(x)|\leq e^{\frac{|\kappa|}{p-1}\,\|u\|_{L^\infty(\Omega)}}\,|f(x)|\quad \text{for a.e. $\in \Omega$.}
$$
Thereby, we can apply Theorem \ref{thm:dirneu} or \ref{thm:conv}, depending on whether Theorem \ref{thm:example1} or \ref{thm:example} is under consideration, 
to the Dirichlet or Neumann problem for the equation \eqref{eqv} and infer that
\begin{equation*}
    \|\nabla u\|_{L^\infty(\Omega)}\leq e^{\frac{|\kappa|}{p-1}\|u\|_{L^\infty(\Omega)}}\,\|\nabla v\|_{L^\infty(\Omega)}\leq c\,e^{{ \frac{p|\kappa|}{(p-1)^2}}\|u\|_{L^\infty(\Omega)}}\,\|f\|^{\frac{1}{p-1}}_{X_n(\Omega)}, 
\end{equation*}
where $c$ is the constant appearing in \eqref{stima:dirichlet} or \eqref{stima:convex}, with $b(t)=t^{p-1}$.
The inequality \eqref{stima:ex1} is thus established.
\end{proof}

\bigskip{}{}

 \par\noindent {\bf Acknowledgments.} C.A. Antonini is a postdoctoral  fellow of the National Institute for Advanced Mathematics (INdAM) at the University of Florence. Part of this work was carried out while he was a postdoctoral fellow at the University of Parma.

\bigskip{}{}

 \par\noindent {\bf Data availability statement.} Data sharing not applicable to this article as no datasets were generated or analysed during the current study.

\section*{Compliance with Ethical Standards}\label{conflicts}

\smallskip
\par\noindent
{\bf Funding}. This research was partly funded by:
\\ (i) GNAMPA   of the Italian INdAM - National Institute of High Mathematics (grant number not available)  (C.A. Antonini , A. Cianchi);
\\ (ii) Research Project   of the Italian Ministry of Education, University and
Research (MIUR) Prin 2017 ``Direct and inverse problems for partial differential equations: theoretical aspects and applications'',
grant number 201758MTR2 (A. Cianchi);
\\ (iii) Research Project   of the Italian Ministry of Education, University and
Research (MIUR) Prin 2022 ``Partial differential equations and related geometric-functional inequalities'',
grant number 20229M52AS, cofunded by PNRR (A. Cianchi);

\bigskip
\par\noindent
{\bf Conflict of Interest}. The authors declare that they have no conflict of interest.

\end{document}